\documentclass[a4paper,oneside,11pt,reqno]{amsart}

\usepackage{amsthm}
\usepackage{amsfonts}
\usepackage{amsmath}
\usepackage{amssymb}
\usepackage{mathrsfs}
\usepackage{epsfig}
\usepackage{pst-plot}
\usepackage{ifthen}
\usepackage[active]{srcltx}
\usepackage{tikz}
\usepackage{lipsum}
\let\OLDthebibliography\thebibliography
\renewcommand\thebibliography[1]{
	\OLDthebibliography{#1}
	\setlength{\parskip}{0pt}
	\setlength{\itemsep}{0pt plus 0.3ex}
}

\makeatletter
\@namedef{subjclassname@2020}{%
  \textup{2020} Mathematics Subject Classification}
\makeatother

\reversemarginpar  

\numberwithin{equation}{section}

\newtheorem{theorem}{\bf Theorem}[section]
\newtheorem{proposition}[theorem]{\bf Proposition}
\newtheorem{corollary}[theorem]{\bf Corollary}
\newtheorem{lemma}[theorem]{\bf Lemma}

\theoremstyle{remark}
\newtheorem{definition}[theorem]{\bf Definition}
\newtheorem{example}[theorem]{\bf Example}
\newtheorem{remark}[theorem]{\bf Remark}

\newcommand{\beq}{\begin{eqnarray}}
\newcommand{\eeq}{\end{eqnarray}}
\newcommand{\beqn}{\begin{eqnarray*}}
\newcommand{\eeqn}{\end{eqnarray*}}
\newcommand{\rar}{\rightarrow}
\newcommand*{\mf}{\mathbf}

\newcommand*{\Ge}{\geqslant}

\newcommand*{\inp}[2]{\langle{#1},\,{#2} \rangle}
\newcommand*{\Le}{\leqslant}

\newtheorem{step}{Step}

\begin{document}

\title[Dirichlet series kernels]
{Quasi-invariance of the Dirichlet series kernels, analytic symbols and \\ homogeneous operators}

   \author{Sameer Chavan}
   \author{Chaman Kumar Sahu}
   \address{Department of Mathematics and Statistics\\
Indian Institute of Technology Kanpur, India}
   \email{chavan@iitk.ac.in}
\email{chamanks@iitk.ac.in}


\begin{abstract} 
For a scalar matrix $\mf a=(a_{m, n})_{m, n=1}^{\infty},$ the Dirichlet series kernel $\kappa_{\mf a}$ 
is the double Dirichlet series $\kappa_{\mf a}(s, u) = \sum_{m, n =1}^{\infty} a_{m, n}m^{-s} n^{-\overline{u}}$ in the variables $s$ and $\overline{u},$ which is regularly convergent on some right half-plane $\mathbb H_\rho.$
The analytic symbols $A_{n, \mf a} = \sum_{m=1}^{\infty}a_{m, n}m^{-s},$ $n \Ge 1$ play a central role in the study of the reproducing kernel Hilbert space $\mathscr H_{\mf a}$ associated with the positive semi-definite kernel $\kappa_{\mf a}.$ 
In particular, they form a total subset of $\mathscr H_{\mf a}$ and provide the formula $\sum_{n=1}^{\infty}\inp{f}{A_{n, \mf a}}n^{-s},$ $s \in \mathbb H_\rho,$ for $f \in \mathscr H_{\mf a}.$ 
We combine the basic theory of Dirichlet series kernels with the Gelfond-Schneider theorem (Hilbert's seventh problem) to show that any quasi-invariant Dirichlet series kernel 
$\kappa_{\mf a}(s, u)$ factors as $f(s)\overline{f(u)}$ for some Dirichlet series $f$ on $\mathbb H_\rho.$ 
In particular, there is no quasi-invariant Dirichlet series kernel $\kappa_{\mf a}$ if the dimension of $\mathscr H_{\mf a}$ is bigger than one. 
This is in strict contrast with the case of the unit disc, where non-factorable quasi-invariant kernels exist in abundance.
 We further discuss the Dirichlet series kernels $\kappa_{\mf a}$ invariant under the group $\mathscr T$ of translation automorphisms of $\mathbb H_\rho$ and construct a family of densely defined $\mathscr T$-homogeneous operators in $\mathscr H_{\mf a},$ whose adjoints are defined only at the zero vector.
\end{abstract}

\maketitle


\section{Preliminaries}

In this section, we collect preliminaries related to the general Dirichlet series and explain the role of the analytic symbols associated with the Dirichlet series kernels in the study of the reproducing kernel Hilbert spaces of Dirichlet series. 

	\subsection{Dirichlet series in one and two variables}
	Denote by $\mathbb N$ and $\mathbb Z_+,$ the sets of positive integers and non-negative integers, respectively.  
	Let $\mathbb Q,$ $\mathbb R$ and $\mathbb C$ denote the sets of rational, real and complex numbers, respectively. 
		For $s \in \mathbb C,$ let $\Re(s),$ $\Im(s),$ $\overline{s},$ $|s|,$ and $\arg(s)$ denote the real part, the imaginary part, the complex conjugate, the modulus and the argument of $s,$ respectively.
		The open unit disc $\{z \in \mathbb C : |z| < 1\}$ is denoted by $\mathbb D,$ while the open upper half-plane $\{z: \Im(z) > 0\}$ is denoted by $\mathbb H.$
	For $\rho \in \mathbb R,$ let $\mathbb H_\rho$ denote the open right half-plane $\{s \in \mathbb C : \Re(s) > \rho\}.$
	By a {\it domain}, we understand a nonempty open subset of $\mathbb C.$ 
For a domain $\Omega,$ let 
$\text{Hol}(\Omega)$ denote the linear space of complex-valued holomorphic functions on $\Omega$ and 
$\mathrm{Aut}(\Omega)$ denote the group of automorphisms  of $\Omega.$	
	
A \textit{general Dirichlet series} is a series of the form
$
 f(s) = \sum_{n=1}^{\infty} a_n e^{-\lambda_ns},$
 where $a_n$ are complex numbers and $\{\lambda_n\}_{n \Ge 1}$ is a sequence of non-negative exponents which increases to $+\infty$ (the reader is referred to \cite{HR} for the basic theory of the general Dirichlet series). The {\it abscissa of absolute convergence} of the series $f(s) = \sum_{n=1}^{\infty} a_{n} e^{-\lambda_n s}$ is the extended real number given by
	$$\sigma_{a}(f) = \inf \big\{\Re(s) : f(s) ~\mbox{is absolutely convergent} \big\},$$
	where we used the conventions that the infimum of an empty set is $\infty$ and infimum of a set not bounded from below is $-\infty.$
\begin{remark} \label{D-S-rmk-0}
If the general Dirichlet series 
$f(s)=\sum_{n=1}^{\infty}\alpha_n e^{-\lambda_n s}$ is convergent in $\mathbb H_\rho,$ then by \cite[Theorem~9]{HR}, the abscissa $\sigma_a(f)$ of the absolute convergence of $f$ satisfies $$\sigma_a(f)  \Le \rho + \limsup_{n \rar \infty}\frac{\log(n)}{\lambda_n}.$$ 
If $\lambda_n:=\omega \log(n),$ $n \Ge 1$ for some $\omega >0,$ then $\sigma_a(f)  \Le \rho + \omega^{-1}.$ 
\end{remark}

The general Dirichlet series is uniquely determined in the following sense. 
	\begin{proposition}[Uniqueness of the general Dirichlet series]  \label{uds} 
For $\rho \in \mathbb R,$ suppose that the general Dirichlet series $f(s)=\sum_{n = 1}^{\infty} a_{n} e^{-\lambda_n s}$ converges on $\mathbb H_{\rho}.$ If $f = 0$ on $\mathbb H_{\rho},$ then $a_{n} = 0$ for every integer $n \Ge 1.$
	\end{proposition}
	\begin{proof}
	Since $\sum_{n=1}^{\infty} a_n e^{-\lambda_ns}$ converges at $s = s_0$
	if and only if the series $ \sum_{n=1}^{\infty} a_n e^{-\lambda_n s_0} e^{-\lambda_n s}$ converges at $s=0,$ the desired uniqueness is immediate from \cite[Theorem~6]{HR}.
	\end{proof}
	 A {\it Dirichlet series} is a series of the form
	$f(s) = \sum_{n=1}^{\infty} a_{n} n^{-s},$
	where $a_n$ are given complex numbers and $s$ is the complex variable. 
Note that the Dirichlet series is obtained from the general Dirichlet series by letting $\lambda_n = \log(n),$ $n \Ge 1.$	
	If $f$ is convergent at $s=s_{0}$, then it converges 
	uniformly throughout the angular region $\{s \in \mathbb C : |\arg(s-s_0)| < \frac{\pi}{2} -\delta\}$ for every positive real number $\delta < \frac{\pi}{2}.$ Consequently, the series $f$ defines a holomorphic function on $\mathbb H_{\Re(s_0)}$ (refer to \cite[Chapter~IX]{Ti} for an excellent exposition on the theory of Dirichlet series).
A {\it Dirichlet polynomial} is a Dirichlet series $f(s)$ given by 
$f(s) = \sum_{n=1}^{N} a_{n} n^{-s}$
for some $N \in \mathbb N.$ 
	
%
	

	A {\it double Dirichlet series} is a series of the form
		\begin{equation*}
			\sum_{m, n =1}^{\infty} a_{m, n} m^{-s} n^{-u},
		\end{equation*}
	where $a_{m, n}$ are given complex numbers and $s, u$ are the complex variables.
	In several variables, the following notion of convergence introduced by Hardy serves our purpose (see also \cite{CGM}). 
		A double Dirichlet series 
		\begin{equation*}
		f(s, u) =	\sum_{m, n =1}^{\infty} a_{m, n} m^{-s} n^{-u}
		\end{equation*}
		is {\it regularly convergent  at $(s_0,u_0) \in \mathbb{C}^{2}$} if the series $f(s, u)$ is convergent at $(s_0,u_0)$, $\sum_{m=1}^{\infty} a_{m, n} m^{-s_0}$ and $\sum_{n =1}^{\infty} a_{m, n} n^{-u_0}$ are convergent for all integers $m, n \Ge 1.$ For 
	$\rho_1, \rho_2 \in \mathbb R,$	we say that $f(s, u)$ is {\it regularly convergent on $\mathbb H_{\rho_1} \times \mathbb H_{\rho_2}$} if it is regularly convergent  at 
	every point	$(s_0, u_0) \in \mathbb H_{\rho_1} \times \mathbb H_{\rho_2}.$

The following lemma is related to the product of two general Dirichlet series and it will be required in Section~\ref{4}.
\begin{lemma} \label{PR-GE-DI}
	Let $f$ and $g$ be two general Dirichlet series given by
	\beqn
	f(s) = \sum_{m = 1}^{\infty} a_{m} e^{-\lambda_ms}  \quad and \quad g(s) = \sum_{m = 1}^{\infty} b_{m} e^{-\mu_ms},
	\eeqn
	which are absolutely convergent on $\mathbb H_\rho.$ 
	If the function $\varphi: \mathbb N \times \mathbb N \rar \mathbb R$ defined by $\varphi(m,n) = \lambda_m+\mu_n,$ $m, n \in \mathbb N,$ is  injective, then 
	\beqn 
	f(s)g(s) = \sum_{q = 1}^{\infty}a_{m_q}b_{n_q} e^{-\nu_qs}, \quad s \in \mathbb H_\rho,
	\eeqn
	where $\{\nu_q\}_{q \Ge 1}$ is a sequence formed by putting the values of $\lambda_m +\mu_n$ in the ascending order and $\varphi^{-1}(\{\nu_q\}) = (m_q, n_q),~ q \Ge 1.$
\end{lemma}
\begin{proof}
Note that injectivity of $\varphi$ implies that for every integer $q \Ge 1,$ the cardinality of $\{(m,n) \in \mathbb N \times \mathbb N: \varphi(m,n) = \nu_q\}$ is one, and hence
	\beqn
	 \sum_{(m,n) \in \varphi^{-1}(\{\nu_q\})}a_{m} b_n = a_{m_q}b_{n_q},
	\eeqn
	where $\varphi^{-1}(\nu_q) = (m_q, n_q).$ The desired conclusion is now immediate from \cite[Theorem~53]{HR} (see also the discussion prior to \cite[Theorem~53]{HR}).
\end{proof}

We record some known facts about double Dirichlet series for later use.  
\begin{lemma} \label{D-S-rmk}
Let $f(s, u) = \sum_{m, n =1}^{\infty} a_{m, n} m^{-s} n^{-u}$ be a double Dirichlet series. Then the following statements are valid$:$
	\begin{enumerate}
\item[$(i)$] if $f(s, u)$ is regularly convergent at some $(s_0, u_0) \in \mathbb C^2,$ then it is absolutely convergent  on $\mathbb H_{\Re(s_0)+1} \times \mathbb H_{\Re(u_0)+1},$ 
\item[$(ii)$] if $f$ is regularly convergent at $(s_0, u_0),$ then the series $f$ defines a holomorphic function in $\mathbb H_{\Re(s_0)} \times \mathbb H_{\Re(u_0)},$ 
\item[$(iii)$] if $f$ is regularly convergent and zero on $\mathbb H_{\rho} \times \mathbb H_\rho,$ then $a_{m, n} = 0$ for all integers $m, n \Ge 1.$
\end{enumerate}
	\end{lemma}
	\begin{proof} 
The part (i) is recorded in \cite[Remark~5]{BCGMS}.
To see (ii), note that if $f(s, u)$ 
		is regularly convergent at $(s_0,u_0),$ then 
		 it converges regularly and 
	uniformly throughout the region $$\big\{(s, u) \in \mathbb C^2 : |\arg(s-s_0)| < \frac{\pi}{2} -\delta, ~|\arg(u-u_0)| < \frac{\pi}{2} -\delta \big\}$$ for every positive real number $\delta < \frac{\pi}{2}$ (see \cite[Theorem~1.8]{CGM}). The desired conclusion in (ii) is now immediate.	 Finally, (iii) follows from (i), Fubini's theorem, and two applications of Proposition~\ref{uds}.
\end{proof}

\subsection{Analytic symbols}

One of the objects of interests in this paper is the  Dirichlet series kernel $\kappa_{\mf a}$ with the coefficient matrix $\mf a$ (see Definition~\ref{def-dsk}). If $\kappa_{\mf a}$ is positive semi-definite, then the Moore's theorem (see \cite[Theorem~2.14]{PR}) yields a reproducing kernel Hilbert space $\mathscr H_{\mf a}$ with the reproducing kernel $\kappa_{\mf a}.$ 
In case of diagonal matrix $\mf a,$ these spaces have been studied extensively in the literature in the last two decades in various contexts 
(see, for example, \cite{BCGMS, CGM, Mc-0, MS}). 
The purpose of this paper is to discuss the role of the so-called analytic symbols in the study of $\mathscr H_{\mf a}.$   
The {\it analytic symbol} $A_{n, \mf a}$ of $\mf a=(a_{m, n})_{m, n =j}^{\infty},$ $j \Ge 1$ are the Dirichlet series given by
\beqn
A_{n, \mf a}(s) = \sum_{m=j}^{\infty} a_{m, n} m^{-s}, \quad s \in \mathbb H_{\rho}, ~n \Ge j.
\eeqn
In this paper, we show that the analytic symbols $A_{n, \mf a},$ $n \Ge 1,$ appear in several contexts in the study of the reproducing kernel Hilbert space $\mathscr H_{\mf a}.$ In particular, they play the role of a canonical ``basis" for the reproducing kernel Hilbert space $\mathscr H_{\mf a}.$ We combine this fact with some results from the analytic and transcendental number theory  (see \cite{Ap, HR, Ni, Tu}) 
to study the Dirichlet series kernels invariant under various subgroups of the automorphism group of a right half-plane.

Here is the plan of the paper. 
In Section~\ref{Sec2}, we characterize positive semi-definite Dirichlet series kernels (see Theorem~\ref{char-pdsk}). 
Among various applications, we give sufficient and necessary conditions for a Dirichlet series to be a member of $\mathscr H_{\mf a}$ (see Corollary~\ref{member}). 
We also show that
$\mathscr H_{\mf a}$ consists only of Dirichlet series. Further, the analytic symbols of the coefficient matrix $\mf a$ form a total subset of $\mathscr H_{\mf a}$ (see Theorem~\ref{total}). 
In Section~\ref{Sec3}, we combine the results from Section~2 with the Schur complement and a variant of Weyl's inequality to produce a non-diagonal reproducing kernel Hilbert space for which all but finitely many analytic symbols are Dirichlet polynomials. 
In Section~\ref{4}, we show that there are no quasi-invariant Dirichlet series kernels on the right half-plane unless the associated reproducing kernel Hilbert spaces are of dimension at most one
(see Theorem~\ref{MA-TH}). 
The proof of this result relies on the results obtained in Section~\ref{Sec2} and solution to the Hilbert's seventh problem (see \cite{Ni, Tu}).
As a consequence of Theorem~\ref{MA-TH}, we show that there are no non-constant $\mathrm{Aut}(\mathbb H_\rho)$-invariant Dirichlet series kernels (see Corollary~\ref{no-aut-linear}).
In Section~\ref{5}, we discuss the translation-invariant Dirichlet series kernels and formally introduce the class of $\mathscr T$-homogeneous operators (see Proposition~\ref{prop-t-inv} and Definition~\ref{t-inv-def}). The main result of this section exhibits a family of densely defined $\mathscr T$-homogeneous operators whose adjoints are defined only at the zero vector (see Theorem~\ref{t-inv-thm}).

 \section{Reproducing kernel Hilbert spaces of Dirichlet series} \label{Sec2}


We formally introduce the notion of the Dirichlet series kernel.

 \begin{definition} \label{def-dsk}
 Let $\mathbf a = (a_{m, n})_{m,n = 1}^{\infty}$ be a matrix with complex entries. 
We say that $\kappa_{\mathbf a}$ given by
\begin{equation*}
\kappa_{\mathbf a}(s,u) = \sum_{m, n =1}^{\infty} a_{m, n} m^{-s} n^{-\bar{u}}
\end{equation*} 
is a {\it Dirichlet series kernel} if there exists $\rho \in \mathbb R$ such that the double Dirichlet series $(s, u) \mapsto \kappa_{\mathbf a}(s, \bar{u})$ is regularly convergent on $\mathbb H_{\rho} \times \mathbb H_{\rho}.$
We refer to $\mf a$ as the {\it coefficient matrix} of the Dirichlet series kernel $\kappa_{\mf a}.$
\end{definition}

Let $X$ be a nonempty set and $f : X \times X \rar \mathbb C$ be a kernel function. We say that $f$ is {\it positive semi-definite} if for any finitely many points $x_1, \ldots, x_n \in X,$ the matrix $(f(x_i, x_j))_{i, j=1}^n$ is positive semi-definite, that is, it is self-adjoint with eigenvalues contained in $[0, \infty).$ A matrix $\mathbf a = (a_{m, n})_{m,n = 1}^{\infty}$ with complex entries is said to be {\it formally positive semi-definite} if the kernel function $f(m, n)=a_{m, n},$ $m, n \Ge 1,$ is positive semi-definite.

Consider a positive semi-definite Dirichlet series kernel $\kappa_{\mathbf a} : \mathbb H_{\rho} \times \mathbb H_{\rho} \rar \mathbb C.$ By \cite[Theorem~2.14]{AMY}, there exists a unique complex Hilbert space $\mathscr H_{\mathbf a}$ such that  $\kappa_{\mathbf a, t} \in \mathscr H_{\mathbf a}$ and
\beq 
\notag
\mathscr H_{\mathbf a} &=& \bigvee \{\kappa_{\mathbf a, t} : t \in \mathbb H_\rho\}, \\ \label{rp}
\inp{f}{\kappa_{\mathbf a, t}} &=& f(t), \quad t \in \mathbb H_\rho, \, f \in \mathscr H_{\mf a},
\eeq
where $\vee$ stands for the closed linear span and
$\kappa_{\mathbf a, t}(s)=\kappa_{\mathbf a}(s, t),$ $s, t \in \mathbb H_\rho.$ 
We refer to the space $\mathscr H_{\mathbf a}$ as the {\it reproducing kernel  Hilbert space associated with $\kappa_{\mathbf a}.$}  Sometimes, the inner-product on $\mathscr H_{\mf a}$ will be denoted by $\inp{\cdot}{\cdot}_{\mathscr H_{\mf a}}.$
\begin{remark} \label{rmk-H-k}
We note the following$:$
\begin{enumerate}
\item[$(i)$]  If $\mathbb L$ denotes a countable dense subset of $\mathbb H_\rho,$ then by the reproducing property of $\mathscr H_{\mathbf a}$ $($see \eqref{rp}$),$
 $\mbox{span}_{\mathbb Q}\{\kappa_{\mathbf a, t} : t \in \mathbb L\}$ is a countable dense subset of $\mathscr H_{\mathbf a},$ where $\mbox{span}_{\mathbb Q}$ stands for the $\mathbb Q$-linear span.
 In particular, $\mathscr H_{\mathbf a}$ is separable.
 \item[$(ii)$] 
By \eqref{rp}, the convergence in $\mathscr H_{\mathbf a}$ implies the uniform convergence on every compact subset of $\mathbb H_\rho.$ Since $\kappa_{\mathbf a, t}$ is holomorphic on $\mathbb H_\rho$ $($consult Lemma~\ref{D-S-rmk}$(ii)),$ every function $f \in \mathscr H_{\mathbf a},$ being a limit of a sequence in $\mbox{span}_{\mathbb Q}\{\kappa_{\mathbf a, t} : t \in \mathbb L\},$ 
is holomorphic on $\mathbb H_\rho.$
 \end{enumerate}
 \end{remark}

The following theorem characterizes positive semi-definite Dirichlet series kernels (cf. \cite[Lemma~4.1]{CS} and \cite[Lemma~20]{MS}). 
\begin{theorem} \label{char-pdsk}
	Let $\mathbf a = (a_{m, n})_{m,n = 1}^{\infty}$ be a matrix with complex entries and let $\kappa_{\mathbf a}$ be a Dirichlet series kernel. Then the following statements are equivalent$:$ 
	\begin{enumerate}
		\item[$(i)$] the Dirichlet series kernel $\kappa_{\mf a}$ is positive semi-definite,
		\item[$(ii)$] the coefficient matrix $\mathbf a$ is formally positive semi-definite.
	\end{enumerate}
	If $(\text{i})$ holds, then for every positive integer $n,$ the analytic symbol
	$A_{n, \mf a}$
	belongs to $\mathscr H_{\mf a}$ and satisfies
	\beq
	\label{eq-gram}
	a_{m, n} = \inp{A_{n, \mf a}}{A_{m, \mf a}}_{\mathscr H_{\mf a}},  \quad m, n \Ge 1.
	\eeq
\end{theorem}

The non-trivial part in the proof of Theorem~\ref{char-pdsk} is the fact that the analytic symbols $A_{n, \mf a},$ $n \Ge 1,$ belong to $\mathscr H_{\mf a}.$ Given this, it is not difficult to obtain \eqref{eq-gram} from Schnee formula (see \cite[Theorem~11.17]{Ap}). Both these facts can be derived simultaneously by a strong induction argument and this occupies the major portion of the proof of Theorem~\ref{char-pdsk}
(see Lemma~\ref{grammian}).  
We begin with the following simple fact, which is a consequence of the uniqueness of the Dirichlet series.

\begin{lemma} \label{self-ad}
	Assume that $\kappa_{\mathbf a} : \mathbb H_{\rho} \times \mathbb H_{\rho} \rar \mathbb C$ is a positive semi-definite Dirichlet series kernel. Then $\mathbf a$ is a self-adjoint matrix.
\end{lemma}
\begin{proof} 
	Since  $\kappa_{\mf a}(s, u) = \overline{\kappa_{\mf a}(u, s)},$ $s, u \in \mathbb{H}_\rho,$ we may conclude that 
	\beqn
	\eta_{\mf a}(s, u):=\sum_{m, n = 1}^{\infty}  (a_{m,n}-\overline{a_{n,m}}) m^{-s} n^{-\overline{u}} =0,  \quad s, u \in \mathbb{H}_{\rho}.
	\eeqn
	Applying Lemma~\ref{D-S-rmk}(iii) to the double Dirichlet series $(s, u) \mapsto \eta_{\mf a}(s, \overline{u})$ completes the proof. 
\end{proof}

We also need a counterpart of \cite[Lemma~11.1]{Ap} for double Dirichlet series.

\begin{lemma} 
	For positive integers $k, l,$ consider the double Dirichlet series 
	\beqn
	f(s,u) = \sum_{m = k}^{\infty}\sum_{n=l}^{\infty} a_{m,n} m^{-s}n^{-u}.
	\eeqn 
	Assume that $f$ is  absolutely convergent on $\mathbb H_{\rho_{1}} \times \mathbb H_{\rho_{2}}$ for some $\rho_1, \rho_2 \in \mathbb R.$  If $r > \max\{\rho_{1}, \rho_{2}\},$ then there exist positive constants $C_{r}$ and $D_r$ such that for every $s, u \in \mathbb H_{r},$ we have
	\beq \label{vanish-infty-1}
	& & \big|k^{s}l^uf(s,u) - a_{k,l}\big| \\   &\Le &  \notag
	C_r\Big(\frac{l^{\Re(u)}}{(l+1)^{\Re(u)-r}} + \frac{k^{\Re(s)}}{(k+1)^{\Re(s)-r}} + \frac{k^{\Re(s)} l^{\Re(u)}}{(k+1)^{\Re(s)-r}(l+1)^{\Re(u)-r}}\Big), \\
	\label{vanish-infty-2}
	& & \Big|l^{u}f(s,u) - \sum_{m=k}^{\infty} a_{m, l} m^{-s}\Big| \\  & \Le  &  \notag
	D_r\Big(\frac{l^{\Re(u)}}{k^{\Re(s)}(l+1)^{\Re(u)-r}} + \frac{ l^{\Re(u)}}{(k+1)^{\Re(s)-r}(l+1)^{\Re(u)-r}}\Big). 
	\eeq
\end{lemma}
\begin{proof}
	Let $r > \max\{\rho_{1}, \rho_{2}\}.$ Note that for any $s, u \in \mathbb H_r,$
	\allowdisplaybreaks
	\beqn 
	&&	|k^{s}l^u f(s,u) - a_{k,l}| 
	\\ & \Le & 
	k^{\Re(s)}l^{\Re(u)} \sum_{\substack{m \Ge k, n \Ge l \\ (m,n) \neq (k,l)}} |a_{m,n}| m^{-\Re(s)} n^{-\Re(u)}
	\\ &=& l^{\Re(u)}\sum_{n \Ge l+1} |a_{k,n}| n^{-\Re(u)} + k^{\Re(s)} \sum_{m \Ge k+1} |a_{m,l}| m^{-\Re(s)} 
	\\ &+& k^{\Re(s)}l^{\Re(u)}  \sum_{m \Ge k+1, n \Ge l+1} |a_{m,n}| m^{-\Re(s)} n^{-\Re(u)}
	\\ &\Le& \frac{l^{\Re(u)}}{(l+1)^{\Re(u)-r}}\sum_{n \Ge l+1} |a_{k,n}| n^{-r} + \frac{k^{\Re(s)}}{(k+1)^{\Re(s)-r}}\sum_{m \Ge k+1} |a_{m,l}| m^{-r}
	\\ &+& \frac{k^{\Re(s)} l^{\Re(u)}}{(k+1)^{\Re(s)-r}(l+1)^{\Re(u)-r}} \sum_{m \Ge k+1, n \Ge l+1} |a_{m,n}| m^{-r} n^{-r}.
	\eeqn
	Since $f$ is absolutely convergent, all the series appearing in the last inequality are convergent, and hence we obtain the constant $C_r$ satisfying \eqref{vanish-infty-1}. To see the estimate \eqref{vanish-infty-2}, note that for any $s, u \in \mathbb H_r,$
	\allowdisplaybreaks
	\beqn 
	&&	|l^{u}f(s,u) - \sum_{m=k}^{\infty} a_{m, l} m^{-s}| 
	\\ & \Le & 
	l^{\Re(u)} \sum_{\substack{m \Ge k, n \Ge l+1 }} |a_{m,n}| m^{-\Re(s)} n^{-\Re(u)}
	\\ &=& \frac{l^{\Re(u)}}{k^{\Re(s)}}\sum_{n \Ge l+1} |a_{k,n}| n^{-\Re(u)} + l^{\Re(u)}  \sum_{m \Ge k+1, n \Ge l+1} |a_{m,n}| m^{-\Re(s)} n^{-\Re(u)}
	\\ &\Le& \frac{l^{\Re(u)}}{k^{\Re(s)}(l+1)^{\Re(u)-r}}\sum_{n \Ge l+1} |a_{k,n}| n^{-r}\\
	 &+& \frac{l^{\Re(u)}}{(k+1)^{\Re(s)-r}(l+1)^{\Re(u)-r}} \sum_{m \Ge k+1, n \Ge l+1} |a_{m,n}| m^{-r} n^{-r}.
	\eeqn
	Now one may argue as in the previous estimate to obtain the constant $D_r$ satisfying \eqref{vanish-infty-2}. This completes the proof.
\end{proof}

The following lemma ensures that the coefficient matrix $\mf a$ of a positive semi-definite Dirichlet series kernel $\kappa_{\mf a}$ can be expressed as a Grammian of the sequence of analytic symbols $A_{n, \mf a},$ $n \Ge 1.$
\begin{lemma} \label{grammian}
	Let $\rho$ be a real number and 
	let $\mathscr H_{\mathbf a}$ be the reproducing kernel Hilbert space associated with the positive semi-definite Dirichlet series kernel $\kappa_{\mathbf a} : \mathbb{H}_\rho \times \mathbb{H}_\rho \rar \mathbb C.$ The following statements are valid$:$
	\begin{enumerate}
		\item[$(i)$] for every integer $n \Ge 1,$ the analytic symbol
		$A_{n, \mf a}$ belongs to $\mathscr H_{\mf a},$ 
		\item[$(ii)$] for every integer $m, n \Ge 1,$ 
		$a_{m, n} = \inp{A_{n, \mf a}}{A_{m, \mf a}}_{\mathscr H_{\mf a}}.$
	\end{enumerate}
\end{lemma}
\begin{proof} 
	We prove (i) and (ii) by a strong induction on $k \Ge 1$ with the following induction hypothesis:
	\beq \label{ind-hypo}
	A_{n, \mf a} \in \mathscr H_{\mf a}, ~n=1, \ldots, k. \quad a_{m, n} = \inp{A_{n, \mf a}}{A_{m, \mf a}},  ~ m, n = 1, \ldots, k.
	\eeq
	We divide the proof into two steps:
	\begin{step} $A_{1, \mf a} \in \mathscr H_{\mf a}$ and 
		\beq \label{p-e-infinity}
		\inp{g}{A_{1, \mf a}} = \lim_{p \rar \infty}g(p), \quad g \in \mathscr H_{\mf a}.
		\eeq
		In particular, $a_{1, 1} = \inp{A_{1, \mf a}}{A_{1, \mf a}}.$
	\end{step}
	By \eqref{rp}, for all integers $p, q \Ge \rho+1,$
	\beqn
	\|\kappa_{\mf a, p} - \kappa_{\mf a, q}\|^2 &=& \|\kappa_{\mf a, p}\|^2 - 2 \Re (\inp{\kappa_{\mf a, p}}{\kappa_{\mf a, q}}) + \|\kappa_{\mf a, q}\|^2 \\
	&=&    \kappa_{\mf a}(p, p) - 2 \Re (\kappa_{\mf a}(q, p)) + \kappa_{\mf a}(q, q), 
	\eeqn
	which converges to $0$ as $p, q \rar \infty$ (use \eqref{vanish-infty-1} with $k=1$ and $l=1$). Thus, $\{\kappa_{\mf a, p}\}_{p \Ge \rho +1}$ is a Cauchy sequence in $\mathscr H_{\mf a},$ and hence it converges to some $f \in \mathscr H_{\mf a}.$ By Remark~\ref{rmk-H-k}(ii), $\kappa_{\mf a, p}(s) \rar f(s)$ as $p \rar \infty$ for every $s \in \mathbb H_{\rho}.$ Hence, by \eqref{vanish-infty-2} (with $k=1$ and $l=1$),
	\beqn
	f(s)=\sum_{m=1}^{\infty} a_{m1} m^{-s}, \quad s \in \mathbb H_{\rho+1}.
	\eeqn This combined with Remark~\ref{rmk-H-k}(ii) and the identity theorem shows that the above equality holds everywhere on $\mathbb H_\rho.$ This yields that $A_{1, \mf a} \in \mathscr H_{\mf a}.$ To see \eqref{p-e-infinity}, note that 
	by the reproducing property of $\kappa_{\mf a}$ (see \eqref{rp}), for any $g \in \mathscr H_{\mf a},$
	\beqn
	\inp{g}{A_{1, \mf a}} = \inp{g}{\lim_{p \rar \infty}\kappa_{\mf a, p}} = \lim_{p \rar \infty} \inp{g}{\kappa_{\mf a, p}} = \lim_{p \rar \infty}g(p).
	\eeqn
	To get the remaining part of Step 1, let $g=A_{1, \mf a}$ in \eqref{p-e-infinity} and apply \cite[Lemma~11.1]{Ap}.
	
	\begin{step} If \eqref{ind-hypo} holds,
		then $A_{k+1, \mf a} \in  \mathscr H_{\mf a}$ and $a_{k+1, n} = 
		\inp{A_{n, \mf a}}{A_{k+1, \mf a}}$ for every $n=1, \ldots, k+1.$  
	\end{step}
	For every integer $p \Ge \rho + 1,$ note that $g_p := \kappa_{\mathbf a, p} - \sum_{l=1}^k l^{-p} A_{l, \mf a} \in \mathscr H_{\mf a}$ and 
	\beq \label{p-cgn}
	g_p(s) = \sum_{m =1}^{\infty} \sum_{n=k+1}^{\infty} a_{m, n} m^{-s} n^{- p}, \quad s \in \mathbb H_{\rho+1}.
	\eeq
	We claim that $\{(k+1)^{p}g_p\}_{p \Ge \rho + 1}$ is a Cauchy sequence in $\mathscr H_{\mf a}.$
	Note that for any integer $p \Ge \rho + 1,$ by the induction hypothesis \eqref{ind-hypo}, \eqref{rp} and the conjugate symmetry of $\mf a$ (see Lemma~\ref{self-ad}),
	\allowdisplaybreaks
	\beq \label{eq-2} \notag
	\inp{g_q}{g_p} 
	&=& \kappa_{\mathbf a}(p, q)-  \sum_{l=1}^k l^{-q} A_{l, \mf a}(p) - 
	\sum_{l=1}^k l^{-p}\overline{A_{l, \mf a}(q)} \\ 
	&+& \notag \sum_{m, n=1}^k m^{-q} n^{-p}\inp{A_{m, \mf a}}{A_{n, \mf a}}  \\ \notag
	&=& \sum_{m, n =1}^{\infty} a_{m, n} m^{-p} n^{-q} - \sum_{m=1}^{\infty}  \sum_{l=1}^k a_{m, l} l^{-q}  m^{-p} 
	\\ &-& \notag \sum_{m=1}^{\infty}  \sum_{l=1}^k a_{l, m} l^{-p}  m^{-q} + \sum_{m, n=1}^k m^{-q} n^{-p} \overline{a_{m, n}}\\ 
	&=& \sum_{m, n =k+1}^{\infty}  a_{m, n} m^{-p} n^{-q}.
	\eeq
	It follows that for any integers $p, q \Ge \rho+1,$
	\allowdisplaybreaks
	\beqn
	\|(k+1)^{p}g_p - (k+1)^{q}g_q\|^2 &=& (k+1)^{2p} \|g_p\|^2 
	-  (k+1)^{p+q}\inp{g_q}{g_p} \\ &-& (k+1)^{p+q}\inp{g_p}{g_q} + (k+1)^{2q} \|g_q\|^2 \\
	&=& (k+1)^{2p} \sum_{m, n =k+1}^{\infty} a_{m, n} m^{-p} n^{-p}  \\ &-&   (k+1)^{p+q}\sum_{m, n =k+1}^{\infty}  a_{m, n} m^{-p} n^{-q}     \\ &-&  (k+1)^{p+q} \sum_{m, n =k+1}^{\infty}  a_{n, m} m^{-p}  n^{-q}   \\ &+& (k+1)^{2q}  
	\sum_{m, n =k+1}^{\infty} a_{m, n} m^{-q} n^{-q},
	\eeqn
	which converges $0$ as $p, q \rar \infty$ (apply \eqref{vanish-infty-1}). Thus, the claim stands verified. 
	As in Step 1, one can see using \eqref{p-cgn} and \eqref{vanish-infty-2} together with the identity theorem that 
	$A_{k+1, \mf a,}$ being the limit of  $\{(k+1)^p g_p\}_{p \Ge \rho + 1},$ belongs to $\mathscr H_{\mf a}.$  Moreover, for any $q = 1, \ldots, k$ and integer $p \Ge \rho + 1,$
	\allowdisplaybreaks
	\beqn
	\inp{A_{q, \mf a}}{(k+1)^{p}g_p} &=&  
	(k+1)^{p} \Big(\inp{A_{q, \mf a}}{\kappa_{\mathbf a, p}} 
	- \sum_{l=1}^k l^{-p}\inp{A_{q, \mf a}}{A_{l, \mf a}}\Big) \\
	&\overset{\eqref{ind-hypo}}=& (k+1)^{p} \Big(A_{q, \mf a}(p) - \sum_{l=1}^k l^{-p}a_{l, q}\Big) \\
	&=&  (k+1)^{p} \Big(\sum_{m=1}^{\infty} a_{m, q} m^{-p} - \sum_{l=1}^ka_{l, q} l^{-p}\Big)\\
	&=& (k+1)^p \sum_{m=k+1}^{\infty} a_{m, q} m^{-p},
	\eeqn
	which, by \cite[Chapter~11, Lemma~1]{Ap}, converges to $a_{k+1, q}.$ Hence, we obtain
	\beqn
	\inp{A_{q, \mf a}}{A_{k+1, \mf a}}=a_{k+1, q}, \quad q=1, \ldots, k.
	\eeqn 
	Since $\{(k+1)^p g_p\}_{p \Ge \rho + 1}$ converges to $A_{k+1, \mf a}$ in $\mathscr H_{\mf a},$ 
	one can see using \eqref{eq-2} and \eqref{vanish-infty-1} that $\inp{A_{k+1, \mf a}}{A_{k+1, \mf a}}=a_{k+1, k+1},$
	which completes the proof.
\end{proof}

\begin{proof}[Proof of Theorem~\ref{char-pdsk}]
	The implication (i)$\Rightarrow$(ii) and the remaining part follow from Lemma \ref{grammian}.	To see the implication (ii)$\Rightarrow$(i), 
	assume that $\mf a$ is a formally positive semi-definite matrix. For a positive integer $N,$ $\{s_{1}, \ldots, s_{N}\} \subseteq \mathbb H_{\rho}$ and $\{c_{1}, \ldots, c_{N}\} \subseteq \mathbb{C}$, note that
	\allowdisplaybreaks
	\beqn
	\sum_{i,j = 1}^{N} c_{i} \overline{c_{j}} \kappa_{\mf a}(s_{i}, s_{j}) &=& \sum_{m,n = 1}^{\infty} \Big(\sum_{i,j = 1}^{N} c_{i} \overline{c_{j}} m^{-s_{i}} n^{-\overline{s_{j}}}\Big) a_{m, n}\\
	&=&  \lim_{\ell \to \infty} \sum_{m,n = 1}^{\ell}  \Big(\sum_{i= 1}^{N} c_{i} m^{-s_{i}} \Big) \overline{\Big(\sum_{j= 1}^{N} c_{j} n^{-s_{j}} \Big)} a_{m, n} \\
	&=& \lim_{\ell \to \infty} \langle{A_{\ell} X_{\ell}},\,{X_{\ell}}\rangle_{\mathbb C^{\ell}},
	\eeqn
	where $A_{\ell}$ is the $\ell \times \ell$ matrix $(a_{m, n})_{m, n=1}^{\ell}$ and $X_{\ell}$ is the $\ell \times 1$ matrix with $k^{\mbox{\tiny th}}$ entry $\sum_{i= 1}^{N} c_{i} k^{-s_{i}}.$ Since each $A_{\ell}$ is positive semi-definite, we get the desired conclusion.
\end{proof}	

We have already seen that the limit of any element in $\mathscr H_{\mf a}$ at infinity exists (see \eqref{p-e-infinity}). This additional feature of $\mathscr H_{\mf a}$ leads to the following fact (cf. \cite[Theorem~4.5]{PR}). 
\begin{corollary} 
	Let $\mathscr H_{\mathbf a}$ be the reproducing kernel Hilbert space associated with the Dirichlet series kernel $\kappa_{\mathbf a} : \mathbb H_{\rho} \times \mathbb H_{\rho} \rar \mathbb C.$ Then $$\mathscr H_{\mf a, \infty}:=\{f \in \mathscr H_{\mf a} : \lim_{s \rar \infty}f(s)=0\}$$ is a closed subspace of $\mathscr H_{\mathbf a}.$ 
	If $\lim_{t \rar \infty}\lim_{r \rar \infty} \kappa_{\mf a}(t, r) \neq 0,$ then the reproducing kernel $\kappa_{\mathbf a, \infty}$ of 
	$\mathscr H_{\mf a, \infty}$ is given by
	\beqn
	\kappa_{\mathbf a, \infty}(s, u) = \kappa_{\mf a}(s, u) - \frac{ \displaystyle\lim_{r \rar \infty}\kappa_{\mf a}(s, r) 
		\lim_{t \rar \infty}\kappa_{\mf a}(t, u)}{ \displaystyle \lim_{t \rar \infty}\lim_{r \rar \infty} \kappa_{\mf a}(t, r)}, \quad s, u \in \mathbb H_{\rho}.
	\eeqn
\end{corollary}
\begin{proof} 
	By Step~I of Lemma~\ref{grammian}, the analytic symbol $A_{1, \mf a}$ belongs to $\mathscr H_{\mf a},$ 
	\beq \label{kernel-infty}
	\text{$\{\kappa_{\mf a}(\cdot, r)\}_{r \Ge \rho}$ converges to $A_{1, \mf a}$ in $\mathscr H_{\mf a}$},
	\eeq 
	and $\mathscr H_{\mf a, \infty} = \{f \in \mathscr H_{\mf a} : \inp{f}{A_{1, \mf a}}=0\}.$ Thus, $\mathscr H_{\mf a, \infty}$ is a closed subspace of $\mathscr H_{\mathbf a}.$ To see the second part, 
	assume that $\lim_{t \rar \infty}\lim_{r \rar \infty} \kappa_{\mf a}(t, r) \neq 0.$
	Note that
	$\kappa_{\mf a, \infty}=\kappa_{\mf b},$ where 
	\beqn
	\mf b = \Big(a_{m, n} - \frac{a_{m, 1}a_{1, n}}{a_{1, 1}}\Big)_{m, n = 1}^{\infty}.
	\eeqn
	By Cholesky's theorem (see \cite[Theorem~4.2]{PR}), $\mf b$ is formally positive semi-definite, and 
	hence by Theorem~\ref{char-pdsk}, $\kappa_{\mf a, \infty}$ is a positive semi-definite Dirichlet series kernel. It is now easy to see using \eqref{kernel-infty} that $\kappa_{\mf a, \infty}$ is  a reproducing kernel for $\mathscr H_{\mf a, \infty}.$ 
\end{proof}

As a consequence of 
Theorem~\ref{char-pdsk}, we characterize those members of $\mathscr H_{\mf a},$ which are Dirichlet series.
\begin{corollary}  \label{member}
	Let $\mathscr H_{\mathbf a}$ be the reproducing kernel Hilbert space associated with the Dirichlet series kernel $\kappa_{\mathbf a} : \mathbb H_{\rho} \times \mathbb H_{\rho} \rar \mathbb C.$ Let $f(s) = \sum_{n=1}^{\infty} \hat{f}(n) n^{-s}$ be a Dirichlet series convergent at every $s \in \mathbb H_\rho.$ Then 
	$f$ belongs to $\mathscr H_{\mf a}$ if and only if there exists a real number $c \Ge 0$ such that
	$\big(c^2 a_{m, n} - \hat{f}(m)\overline{\hat{f}(n)}\big)_{m, n=1}^{\infty}$
	is formally positive semi-definite.
\end{corollary}
\begin{proof} 
	By \cite[Theorem~3.11]{PR}, $f \in \mathscr H_{\mf a}$ if and only if there exists a  real number $c$ such that 
	\beqn
	\kappa(s, u) := c^2 \kappa_{\mf a}(s, u) - f(s)\overline{f(u)}, \quad s, u \in \mathbb H_\rho 
	\eeqn
	is a positive semi-definite kernel. It is easy to see that $\kappa$ is a Dirichlet series kernel with
	the coefficient matrix $\big(c^2 a_{m, n} - \hat{f}(m)\overline{\hat{f}(n)}\big)_{m, n=1}^{\infty}.$ The desired equivalence now follows from Theorem~\ref{char-pdsk}. 
\end{proof}

We now show that $\mathscr H_{\mf a}$ consists only of Dirichlet series.
\begin{theorem} \label{total}
	Let $\mathscr H_{\mathbf a}$ be the reproducing kernel Hilbert space associated with the Dirichlet series kernel $\kappa_{\mathbf a} : \mathbb H_{\rho} \times \mathbb H_{\rho} \rar \mathbb C.$ 
	If 
	$
	A_{n, \mf a},$ $n \Ge 1$ are the analytic symbols of $\mf a,$ 
	then the following statements are valid$:$
	\begin{enumerate}
		\item[$(i)$] the Dirichlet series kernel $\kappa_{\mf a}$ is given by 
		\beqn
		\kappa_{\mf a}(s, u) = \sum_{n =1}^{\infty} A_{n, \mf a}(s) n^{-\bar{u}}, \quad s, u \in \mathbb H_{\rho},
		\eeqn
		\item[$(ii)$]
		any element $f$ in $\mathscr H_a$ is a Dirichlet series given by
		\beqn
		f(s) = \sum_{n =1}^{\infty} \inp{f}{A_{n, \mf a}}_{\mathscr H_{\mf a}} \, n^{-s}, \quad s \in \mathbb H_\rho,
		\eeqn 
		\item[$(iii)$] $\{A_{n, \mf a}\}_{n = 1}^{\infty}$ is a total set in $\mathscr H_{\mf a}.$
	\end{enumerate}
\end{theorem}
\begin{proof} By Theorem~\ref{char-pdsk}, 
	for every positive integer $n,$ the Dirichlet series
	$A_{n, \mf a}$ belongs to $\mathscr H_{\mf a}.$ We claim that for every $u \in \mathbb H_\rho,$ 
	the series $\sum_{n =1}^{\infty} A_{n, \mf a} n^{-\bar{u}}$ is convergent in $\mathscr H_{\mf a}.$
	Note that for any positive integers $M, N$ such that $M > N$ and $u \in \mathbb H_\rho,$
	\beqn
	\Big\|\sum_{n =N+1}^{M} A_{n, \mf a} n^{-\bar{u}}\Big\|^2 = \sum_{k, l = N+1}^M \inp{A_{k, \mf a}}{A_{l, \mf a}}k^{-\bar{u}}l^{-u} \overset{\eqref{eq-gram}}= \sum_{k, l = N+1}^M a_{l, k}k^{-\bar{u}}l^{-{u}},
	\eeqn
	which converges to $0$ as $M, N \rar \infty$ (since $\kappa_{\mf a}$ is convergent on $\mathbb H_\rho \times \mathbb H_\rho$). This completes the verification of the claim.
	
	(i) 
	Let $u \in \mathbb H_{\rho+1}.$ By Lemma~\ref{D-S-rmk}(i),
	\beqn
	\kappa_{\mf a, u}(s) = \sum_{m, n =1}^{\infty} a_{m, n} m^{-s} n^{-\bar{u}} = \sum_{n =1}^{\infty} A_{n, \mf a}(s) n^{-\bar{u}}, \quad s \in \mathbb H_{\rho+1}.
	\eeqn
	Since $\kappa_{\mf a, u}$ and $\sum_{n =1}^{\infty} A_{n, \mf a} n^{-\bar{u}}$ are holomorphic on $\mathbb H_\rho$ (see Remark~\ref{rmk-H-k}(ii)), by the identity theorem (see \cite{ShSt}), we obtain 
	\beqn
	\kappa_{\mf a, u}(s) = \sum_{n =1}^{\infty} A_{n, \mf a}(s)n^{-\bar{u}}, \quad s \in \mathbb H_{\rho}.
	\eeqn
	By fixing $s \in \mathbb H_\rho$ and varying $u$ over $\mathbb H_{\rho+1},$ another application of the identity theorem yields (i).
	
	(ii) By the reproducing property \eqref{rp} of $\mathscr H_{\mf a}$ and (i), 
	\beqn
	f(s) = \inp{f}{\kappa_{\mf a, s}}= \sum_{n =1}^{\infty} \inp{f}{A_{n, \mf a}}_{\mathscr H_{\mf a}}\, n^{-s}, \quad s \in \mathbb H_{\rho}.
	\eeqn
	
	(iii) By (ii), $f=0$ provided 
	$\inp{f}{A_{n, \mf a}} =0$ for every integer $n \Ge 1.$ 
\end{proof}

\section{When analytic symbols are Dirichlet polynomials} \label{Sec3}

Let $k$ be a non-negative integer and let $\kappa_{\mf a}$ be a Dirichlet series kernel. We say that $\kappa_a$ is a {\it $k$-diagonal kernel} or {\it finite bandwidth $2k$} if $a_{m, n}=0$ for $m, n \Ge 1$ such that $|m-n| > k.$ 

\begin{proposition} Let $\kappa_{\mf a}$ be a positive semi-definite $k$-diagonal Dirichlet series kernel. Then, for every integer $n \Ge 1,$ 
the analytic symbols $A_{n, \mf a}$ of $\mf a$ is a Dirichlet polynomial. Moreover,  
$\mathscr H_{\mf a}$ admits an orthonormal basis consisting of Dirichlet polynomials.
 In addition, if $\mathscr H_{\mf a}$ contains the space $\mathcal D[s]$ of Dirichlet polynomials, then $\mathcal D[s]$ is dense in $\mathscr H_{\mf a}.$ 
 \end{proposition}
 \begin{proof}
 Since $\kappa_{\mf a}$ is $k$-diagonal, the analytic symbol $A_{n, \mf a}=\sum_{m=1}^{\infty} a_{m, n} m^{-s} $ is equal to the Dirichlet polynomial $\sum_{m=N}^{n+k} a_{m, n} m^{-s},$ where $N$ is equal to $\max\{n-k, 1\}.$  Thus, 
by Theorem~\ref{total} and Gram-Schmidt orthonormalization process, $\mathscr H_{\mf a}$ has an orthonormal basis consisting of Dirichlet polynomials. The remaining assertion is now clear.
\end{proof}

It is evident that all analytic symbols $A_{n, \mf a}$ could be Dirichlet polynomials without the kernel $\kappa_{\mf a}$ being $k$-diagonal for any integer $k \Ge 1.$ 
Indeed, for $k \in \mathbb N,$ let $A_k$ denote the positive semi-definite square matrix $A_{k}$ of order $k$ that has all entries equal to $1$ and consider the matrix $\mathbf a$ given by $\mathbf a = \oplus_{k=1}^{\infty}A_{k}.$ 
Clearly, $\kappa_\mf a$ is a positive semi-definite kernel, which is not $k$-diagonal for any integer $k \Ge 1.$
We exhibit below a class of Dirichlet series kernels for which all but finitely many analytic symbols are Dirichlet polynomials. Our construction relies on a variant of Weyl's inequality and the Schur complement (see \cite{B-R}).

\allowdisplaybreaks
For a positive integer $k,$ let $\mf b=(b_{i, j})_{i, j =1}^k$  be a positive semi-definite matrix, $\{c_{k+l}\}_{l \Ge 1}$ be a sequence of complex numbers and
$\{d_{k+l}\}_{l \Ge 1}$ be a sequence of positive real numbers.
Consider the matrix $\mf a$ given by
	\beqn
	\mf a = 
	\begin{pmatrix}
		b_{1,1} & b_{1,2} & \ldots & b_{1,k} & c_{k+1} & c_{k+2} & \ldots  \\
		b_{2,1} & b_{2,2} & \ldots & b_{2,k} & c_{k+1} & c_{k+2} & \ldots  \\
		\vdots & \vdots & \vdots & \vdots & \vdots  & \vdots & \ldots  \\
		b_{k,1} & b_{k,2} & \ldots & b_{k,k} & c_{k+1} & c_{k+2} & \ldots \\
		\overline{c}_{k+1} & \overline{c}_{k+1} & \ldots & \overline{c}_{k+1} & d_{k+1} & 0  & \ldots\\
		\overline{c}_{k+2} & \overline{c}_{k+2} & \ldots &\overline{c}_{k+2} & 0 &  d_{k+2} & \ldots \\
		\vdots & \vdots &  \vdots &\vdots & \vdots & \vdots & \ddots 
	\end{pmatrix}.
\eeqn
Equivalently,  
  $\mf a$ is a block symmetric matrix given by
	\begin{equation*}
	\mf a = 
	\begin{pmatrix}
	\mf b & \mf c \\
	\mf c^* & \mf d
		\end{pmatrix},
	\end{equation*}
	where $\mf c$ is the $k \times \infty$ matrix with all rows equal to $(c_{k+1}, c_{k+2}, \ldots)$ and $\mf d$ is the diagonal matrix with diagonal entries $d_{k+1}, d_{k+2}, \ldots.$  	
	We say that $\mf a$ belongs to the class $\mathscr S_k$ if the following conditions hold:
	\begin{enumerate}
	\item[(C1)] $\mf b$ is a positive semi-definite matrix,
	\item[(C2)] $\displaystyle \sum_{l = 1}^{\infty} \frac{|c_{k+l}|^{2}}{d_{k+l}}$ is finite.
	\end{enumerate}
	\begin{remark}
	The series given in $(C2)$ can be thought of as the maximum eigenvalue of $k^{-1}\mf c \mf d^{-1} \mf c^*.$ 
	\end{remark}
With every $\mf a \in \mathscr S_k,$ we associate the following number:
\beqn 
\mathcal s(\mf a):= \lambda_{\min}(\mf b) - k \sum_{l = 1}^{\infty} \frac{|c_{k+l}|^{2}}{d_{k+l}},
\eeqn
where $\lambda_{\min}(\cdot)$ denotes the minimum eigenvalue.
For a positive integer $m,$ let $\mf e_m=(e_{p, q})_{p, q=1}^{\infty}$ denote the matrix given by
\beqn
e_{p, q} = \begin{cases} 1 & \mbox{if~}p=m, \,q = m, \\
0 & \mbox{otherwise}.
\end{cases}
\eeqn
\begin{lemma}\label{Fu-Ch}
Assume that $\mf a \in \mathscr S_k.$ 
Then the following statements are valid$:$
\begin{enumerate}
\item[$(i)$] if $\mathcal s(\mf a) \Ge 0$ and $\epsilon$ is a real number such that $0 \Le \epsilon \Le \mathcal s(\mf a),$
then for every positive integer $m \Le k,$ 
$\mf a - \epsilon \mf e_{m}$ is a formally positive semi-definite matrix, 
\item[$(ii)$] if $\mathcal s(\mf a) \Ge 0,$ then $\mf a$ is a formally positive semi-definite matrix,
\item[$(iii)$] if $\mathcal s(\mf a) > 0,$  then for every positive integer $m \Ge 1,$  there exists $\epsilon_m > 0$ such that $\mathcal s(\mf a - \epsilon_m \mf e_{m}) > 0.$ In this case,
$\mf a - \epsilon_m \mf e_{m}$ is a formally positive semi-definite matrix.
\end{enumerate}
\end{lemma}
\begin{proof}
(i) Let $j \Ge 1$ and $m \Le k$ be positive integers and let $\epsilon$ be a real number such that $0 \Le \epsilon \Le \mathcal s(\mf a).$
	Let $\mf a_j$ and $\mf e_{m, j}$  denote the $(k+j)\times(k+j)$ truncated matrices of $\mf a$ and $\mf e_m,$ respectively.
	Note that $\mf a_j - \epsilon \mf e_{m, j}$ is rewritten as
	\beqn
	\mf a_j - \epsilon \mf e_{m, j} = 
	\begin{pmatrix}
		\mf b^{(m)}_{j, \epsilon} & \mf c_j  \\
		\mf c^{*}_j & \mf d_j 
	\end{pmatrix},
	\eeqn
	where $\mf b^{(m)}_{j, \epsilon}$ is the leading principal submatrix of $\mf a_j - \epsilon \mf e_{m, j},$ 
	$\mf c_j$ is the matrix of order $k \times j,$
	and $\mf d_j= \text{diag}(d_{k+1}, \ldots , d_{k+j}).$
	Note that $\mf d_j$ is a strictly positive definite matrix and 
	\beqn  
\mf c_j \mf d^{-1}_j\mf c^{*}_j  &=& \Big(\sum_{l = 1}^{j} \frac{|c_{k+l}|^{2}}{d_{k+l}}\Big) \mf 1,
	\eeqn
	where $\mf 1$ denotes the $k \times k$ matrix with all entries $1.$
By \cite[Lemma~2.20]{PR},	 $\mf c_j \mf d^{-1}_j\mf c^{*}_j$ has eigenvalues $k\sum_{l = 1}^{j} \frac{|c_{k+l}|^{2}}{d_{k+l}}$ and $0$ of multiplicity $1$ and $k-1,$ respectively.
It now follows from a variant of the Weyl's inequality (see \cite[Theorem~3.8.2]{B-R}) applied two times that
	\beq 
	\label{St-ine}
	\notag
	\lambda_{\min}(\mf b^{(m)}_{j, \epsilon} - \mf c_j \mf d^{-1}_j \mf c^{*}_j) &\Ge&  
	\lambda_{\min}(\mf b^{(m)}_{j, \epsilon}) - k\sum_{l = 1}^{j} \frac{|c_{k+l}|^{2}}{d_{k+l}}\\
	&\Ge& \lambda_{\min}(\mf b) - \epsilon - k\sum_{l = 1}^{j} \frac{|c_{k+l}|^{2}}{d_{k+l}} \notag \\ 
	&\Ge& \mathcal s(\mf a) - \epsilon. 
	\eeq
	Thus, $\mf b^{(m)}_{j, \epsilon} - \mf c_j \mf d^{-1}_j\mf c^{*}_j$ is positive semi-definite. Since $\mf d_j$ is positive semi-definite, by the Schur complement (see the discussion following the proof of \cite[Theorem~26.2]{B-R}), $\mf a_j - \epsilon \mf e_{m, j}$ is positive semi-definite, 
and we obtain (i).

(ii) If $\mathcal s(\mf a) \Ge 0$ and $\epsilon =0$ then the conclusion in (i) is valid for any integer $m > k$ (since \eqref{St-ine} holds with $\epsilon =0$). This gives (ii).
	
(iii) In view of (i), we may assume that $m$ be a positive integer bigger than $k.$ Assume that $\mathcal s(\mf a) > 0.$ 
Note that for any real number $\epsilon$ with $0 < \epsilon < d_{m},$
\beqn
\mathcal s(\mf a-\epsilon \mf e_m) &=& \lambda_{\min}(\mf b) - k \Big(\sum_{\overset{l = 1}{l \neq m-k}}^{\infty} \frac{|c_{k+l}|^{2}}{d_{k+l}} + \frac{|c_{m}|^{2}}{d_{m}-\epsilon}\Big) \\
&=& \mathcal s(\mf a) + k|c_{m}|^{2}\Big(\frac{1}{d_{m}} - \frac{1}{d_{m}-\epsilon}\Big).  
\eeqn   
Since $\mathcal s(\mf a) > 0,$ there exists $\epsilon_m \in (0, d_m)$ such that $\mathcal s(\mf a-\epsilon_m \mf e_m) >0.$ The desired conclusion now follows from (ii). 
	\end{proof}
	
	The converse of Lemma~\ref{Fu-Ch}(ii) does not hold in general.
	\begin{example}
	Consider the $2\times2$ positive semi-definite matrix $\mf b$ given by
	\beqn
	\mf b = 
	\begin{pmatrix}
		\frac{1}{2} &\frac{1}{\sqrt{6}} \\
		\frac{1}{\sqrt{6}} & \frac{2}{3} 
	\end{pmatrix}.
	\eeqn
It is easy to see that the eigenvalues of $\mf b$ are $\frac{1}{6}$ and $1.$  
Let $\{c_{l+2}\}_{l \Ge 1}$ be the constant sequence with value $1$ and let $\{d_{l+2}\}_{l \Ge 1}$ be the sequence given by $d_{l+2}=4^l,$ $l \Ge 1.$
Since $\sum_{l = 1}^{\infty} \frac{|c_{l+2}|^{2}}{d_{l+2}}$ converges to $\frac{1}{3},$ the 
matrix $\mf a$ given by
	\beqn 
	\begin{pmatrix}
		\mf b &\mf c  \\
	\mf	c^{*} & \mf d   
	\end{pmatrix}
	\eeqn
belongs to $\mathscr S_{2}.$ Since $\lambda_{\min}(\mf b)=1/6,$ we obtain $\mathcal s(\mf a) = \frac{1}{6} - \frac{2}{3} < 0.$ We check that the matrix $\mf a$ is positive semi definite.  
In view of the Schur complement lemma, it suffices to check that for every integer $j \Ge 1,$
$\mf g_{j} = \mf g^{(1)}_j \oplus \mf g^{(2)}_j$ is positive semi-definite, where 
\beqn 
	\mf g^{(1)}_j =  
	\begin{pmatrix}
		\frac{1}{2}-S_{j} &\frac{1}{\sqrt{6}}-S_{j}   \\
		\frac{1}{\sqrt{6}} -S_{j}& \frac{2}{3}-S_{j} 
	\end{pmatrix}, 
\eeqn
with $S_{j} = \sum_{i = 1}^{j} \frac{1}{4^{i}}$ and $\mf g^{(2)}_j$  is the diagonal matrix with diagonal entries $4, 4^2, \ldots, 4^j.$  
Note that $S_j = \frac{1}{3}(1 - \frac{1}{4^j}),$ and hence 
$S_{j} \in [\frac{1}{4},\frac{1}{3}).$ 
Let $\mathrm{Tr}$ and $\det$ denote the trace and determinant, respectively. 
Thus $$\mathrm{Tr}\,\mf g^{(1)}_j= \frac{1}{2} + \frac{2}{3} - 2 S_j > 0.$$ Moreover, since
	$S_j > \frac{1}{\sqrt{6}}\Big(1-\frac{1}{\sqrt{3}}\Big),$ we get
	\beqn
\det \,\mf g^{(1)}_j = (\frac{1}{2}-S_{j}) (\frac{2}{3}-S_{j}) - (\frac{1}{\sqrt{6}}-S_{j})^{2} \Ge  \frac{1}{18}- (\frac{1}{\sqrt{6}}-S_{j})^{2} > 0.
\eeqn
This shows that $\mf g_j$ is positive semi definite.  \hfill $\diamondsuit$
	\end{example}
	
We now exhibit a family of reproducing kernel Hilbert spaces associated with a non-diagonal Dirichlet series kernels (cf. \cite[Theorem~3.1]{AFMP}). 
	\begin{theorem} 
	Let $\mf a \in \mathscr S_k$ be such that $\mathcal s(\mf a)$ is positive. Assume that there exists a real number $\rho >1$ such that $\mf d$ satisfies 
\beq \label{growth-dl}
d_l  = \mathcal{O}(l^{\rho-1}), \quad l \Ge k+1.
\eeq
	Then $\kappa_{\mf a} : \mathbb H_{\rho} \times \mathbb H_{\rho} \rar \mathbb C$ is a positive semi-definite Dirichlet series kernel. Moreover, 
for every integer $n \Ge k+1,$ the analytic symbol $A_{n, \mf a}$ is a Dirichlet polynomial and 
	Dirichlet polynomials form a dense subspace of $\mathscr H_{\mf a}.$
	\end{theorem}
	\begin{proof}
Note that for any integer $p \Ge k+1,$
	\beqn
	\frac{|c_p|^2}{d_p} \Le	 \sum_{l = 1}^{\infty} \frac{|c_{k+l}|^{2}}{d_{k+l}} = \frac{\lambda_{\min}(\mf b)-\mathcal s(\mf a)}{k} \Le \frac{\lambda_{\min}(\mf b)}{k}.
	\eeqn
Thus, we obtain
	\beqn
|c_{p}| \Le \frac{\sqrt{\lambda_{\min}(\mf b)}}{\sqrt{k}}\sqrt{d_p}, \quad p \Ge k+1.
\eeqn
	This combined with \eqref{growth-dl} shows that 
	\beqn
	|a_{m, n}| \Le  C \max \Big\{\sqrt{d_{m}}, ~d_{m}, ~\sqrt{d_{n}}, ~d_{n}\Big\} \Le C m^{\rho-1} n^{\rho-1}, \quad m, n \Ge k+1,
	\eeqn
where $C=\max \Big\{\frac{\sqrt{\lambda_{\min}(\mf b)}}{\sqrt{k}}, 1\Big\}.$	Thus, the Dirichlet series kernel $\kappa_{\mf a}$ converges absolutely on $\mathbb H_{\rho} \times \mathbb H_{\rho}.$
	By Lemma~\ref{Fu-Ch}, $\mf a$ is formally positive semi-definite, and hence by Theorem~\ref{char-pdsk}, $\kappa_{\mf a}$ is positive semi-definite.
	
By Lemma~\ref{Fu-Ch}(iii), for every integer $m \Ge 1,$ there exists $\epsilon_m > 0$ such that
$\mf a - \epsilon_m \mf e_{m}$ is formally positive semi-definite. An application of Corollary~\ref{member} now shows that $m^{-s} \in \mathscr H_{\mf a}$ for every integer $m \Ge 1.$ Thus, $\mathscr H_{\mf a}$ contains the space of Dirichlet polynomials. Moreover, by Theorem~\ref{total}, $A_{n, \mf a}$ belongs to $\mathscr H_{\mf a}$ for every integer $n \Ge 1.$ 
Note that for every integer $n \Ge k+1,$  $a_{m, n}=0$ for all $m \Ge n+1,$ and hence $A_{n, \mf a}$ is a Dirichlet polynomial. Since $\{A_{n, \mf a}\}_{n \Ge 1}$ is a total set in $\mathscr H_{\mf a}$ (see Theorem~\ref{total}), the Dirichlet polynomials are dense in $\mathscr H_{\mf a}.$
	\end{proof}

\section{Quasi-Invariance} \label{4}

Recall that Bergman type-kernels satisfy a natural transformation rule; that is, the quasi-invariance property (see Definition~\ref{def-qi}; refer to \cite{Mi} for more details).  In this section, we are primarily concerned in the study of positive semi-definite Dirichlet series kernels, which admit the quasi-invariance property. 
Some variants of the following notion appeared in \cite{GmGh} and \cite{KM}. 
\begin{definition} \label{def-qi}
	Let $\Omega$ be a domain and let $G$ be a group acting  on $\Omega.$ Let $\mathcal J: G \times \Omega \rar \mathbb{C}$ be a nowhere-vanishing function such that $\mathcal J(g, .)$ is holomorphic on $\Omega$ for every $g \in G.$ Let $\kappa: \Omega \times \Omega \rar \mathbb{C}$ be a positive semi-definite kernel. 
	\begin{enumerate}
	\item[$\bullet$] 	$\kappa$ is {\it $G$-invariant} if 
	\beqn
	 \kappa(g \cdot s, g \cdot u)  = \kappa(s,u) \quad s,u \in \Omega, \, g \in G.
	\eeqn
	\item[$\bullet$] 	$\kappa$ is {\it $G$-quasi-invariant with respect to $\mathcal J$} if 
	\beq\label{QU-DE}
	\mathcal J(g, s) \kappa(g \cdot s, g \cdot u) \overline{\mathcal J(g, u)} = \kappa(s,u) \quad s,u \in \Omega, \, g \in G.
	\eeq
\end{enumerate}		
	If $G=\mathrm{Aut}(\Omega),$ then we refer to the $G$-quasi-invariant kernel as the {\it quasi-invariant kernel}.
\end{definition}

 Before we state the classification theorem for quasi-invariant Dirichlet series kernels, we
determine the automorphism group of the right half plane $\mathbb H_\rho$ using that of the upper half-plane $\mathbb H.$ To see this, consider the conformal map $\eta_\rho: \mathbb H \rar \mathbb H_\rho$ given by $\eta_\rho(z) = -iz+\rho,$ $z \in \mathbb H,$ 
and let {\it $\text{SL}_2(\mathbb R)$} denote the group of $2 \times 2$ real matrices of determinant $1.$ For $A= {(\begin{smallmatrix}a & b \\ c & d \end{smallmatrix}}) \in$ {\it $\text{SL}_2(\mathbb R),$} and define $\phi_A : \mathbb H_\rho \rar \mathbb H_\rho$ by
\beqn
\phi_{A}(s) := \frac{a(s-\rho) - ib}{ic(s-\rho)+d} + \rho, \quad s \in \mathbb H_\rho.
\eeqn
Since $d + ic(s-\rho)=0$ if and only if $s = \frac{d}{c}i+\rho$ with $c \neq 0$ or $c=d=0,$ 
$\phi_A$ defines a biholomorphism on $\mathbb H_\rho.$ Thus, 
\beqn
\mathrm{Aut}(\mathbb H_\rho) = \{\eta_\rho \circ \psi \circ \eta^{-1}_\rho : \psi \in \mathrm{Aut}(\mathbb H)\}.
\eeqn
It is now easy to see using \cite[Theorem~2.4]{ShSt} that
\beqn
	\mathrm{Aut}(\mathbb{H_{\rho}}) = \left\{\phi_{A} :  A  \in {\it \text{SL}_{2}(\mathbb{R})} \right\}.
\eeqn
Let $\mathrm{Aut}_L(\mathbb H_\rho)$ denote the subgroup of $\mathrm{Aut}(\mathbb H_{\rho})$ (the ``linear'' automorphisms) given by
\beqn
\mathrm{Aut}_L(\mathbb H_\rho) = \{\phi_{A} :  A=({\begin{smallmatrix}a & b \\ 0 & d\end{smallmatrix}}) \in  {\it \text{SL}_2(\mathbb R)} \}.
\eeqn
The matrices $({\begin{smallmatrix}1 & b \\ 0 & 1\end{smallmatrix}}),$ $b \in \mathbb R,$ yield the so-called translation automorphisms$:$
\beq \label{phi-b}
\phi_{b}(s) = s - ib, \quad s \in \mathbb H_\rho.
\eeq
We denote by $\mathscr T,$ the group of all translation automorphisms of $\mathbb H_\rho.$
 
The quasi-invariant kernels on the unit disc are in abundance. In fact, the Bergman kernel $B_\Omega$ of any bounded domain $\Omega$ is quasi-invariant with respect to $$\mathcal J(\psi, z) = \psi'(z), \quad \psi \in \text{Aut}(\Omega), \, z \in \Omega$$ (see \cite[Proposition~1.4.12]{K}). In contrast to this, there are no quasi-invariant Dirichlet series kernels  in dimension bigger than $1.$ Indeed, we have the following$:$
\begin{theorem}\label{MA-TH} Let $\mathbf a = (a_{m, n})_{m,n = 1}^{\infty}$ be a matrix with complex entries and $\kappa_{\mf a}  : \mathbb H_{\rho} \times \mathbb H_{\rho} \rar \mathbb C$ be the Dirichlet series kernel with the coefficient matrix $\mf a.$ Then 
the following statements are equivalent$:$
\begin{enumerate}
\item[$(i)$] 	$\kappa_{\mf a}$ is quasi-invariant with respect to some $\mathcal J,$
\item[$(ii)$]  $\kappa_{\mf a}$ is $\mathrm{Aut}_L(\mathbb H_\rho)$-quasi-invariant with respect to some $\mathcal J,$
\item[$(iii)$] there exists a Dirichlet series $f:\mathbb H_\rho \rar \mathbb C$ such that 
	\beqn					
	\kappa_\mf a(s,u) =  f(s)\overline{f(u)}, \quad s, u \in \mathbb H_\rho,
\eeqn
where $f$ is either identically zero or nowhere-vanishing,
\item[$(iv)$] $\kappa_{\mathbf a}$ is positive semi-definite and the reproducing kernel Hilbert space $\mathscr H_{\mf a}$ associated with $\kappa_{\mf a}$ is either $\{0\}$ or spanned by a nowhere-vanishing Dirichlet series on $\mathbb H_\rho.$
\end{enumerate}	
\end{theorem}

In order to prove Theorem~\ref{MA-TH}, we need several lemmas. The first one asserts that 
any $\mathrm{Aut}_L(\mathbb H_\rho)$-quasi-invariant Dirichlet series kernel is either nowhere-vanishing or identically zero (cf. \cite[Theorem~11.4]{Ap}).
\begin{lemma}\label{EX-PR} 
	Let $\kappa_{\mf a}: \mathbb H_{\rho} \times \mathbb H_{\rho} \rar \mathbb{C}$ be a Dirichlet series kernel with the coefficient matrix $\mf a=(a_{m,n})_{m,n = 1}^{\infty}.$ Assume that $\kappa_{\mf a}$ is a non-zero function. Then the following statements are valid$:$
\begin{enumerate}
\item[$(i)$] there exists a real number $r > \rho+1$ such that $\kappa_{\mf a}$ never vanishes on $\mathbb H_{r} \times \mathbb H_{r},$
 \item[$(ii)$] if $\kappa_{\mf a}$ is $\mathrm{Aut}_L(\mathbb H_\rho)$-quasi-invariant with respect to some $\mathcal J,$ then $\kappa_\mf a$ is nowhere-vanishing.
\end{enumerate}	
\end{lemma}
\begin{proof}
(i):	Suppose that the conclusion is not true. Thus, there exists a sequence $\{(s_{n}, u_{n})\}_{n=1}^{\infty} \subseteq \mathbb H_{\rho} \times \mathbb H_{\rho}$ such that $\{\Re(s_{n})\}_{n=1}^{\infty}, \{\Re(u_{n})\}_{n=1}^{\infty}$ are unbounded and 
	\beq
	\label{vani-seq}
	\kappa_{\mf a}(s_{n}, u_{n}) = 0, \quad n \Ge 1. 
	\eeq
	Since $\kappa_{\mf a}$ is absolutely convergent on $\mathbb H_{\rho +1} \times \mathbb H_{\rho +1},$	an application of \eqref{vanish-infty-1} shows that for any positive integers $i$ and $j,$
	\beqn
	a_{i,j} = \lim_{n \rar \infty} i^{s_{n}}j^{u_{n}}\kappa_{\mf a}(s_{n}, u_{n}) \overset{\eqref{vani-seq}}=0.
	\eeqn
	This shows that $\mf a =0.$

(ii): Suppose that $\kappa_\mf a$ vanishes at some point $(s_0, u_0) \in \mathbb H_\rho \times \mathbb H_\rho.$
	Since $s \rar \mathcal J(\psi, s)$ is a nowhere-vanishing function, by \eqref{QU-DE}, 
\beq \label{zero-kappa}
\kappa_\mf a(\psi_a(s_0), \psi_a(u_0)) = 0, \quad a \in \mathbb R \backslash \{0\},
\eeq 
 where $\psi_a(s) = a^2(s-\rho)+\rho,$ $s \in \mathbb H_\rho.$ 
Since 
$\lim_{a \rar \infty}\psi_a(s) = \infty$ for any $s \in \mathbb H_\rho,$ 
for every $\rho' \Ge \rho,$ there exists $a \in \mathbb R\backslash \{0\}$ such that $\psi_a(s_0), \psi_a(u_0) \in \mathbb H_{\rho'}.$ This combined with \eqref{zero-kappa} implies that 
$\kappa_{\mf a}$ admits a zero in $\mathbb H_{\rho'} \times \mathbb H_{\rho'}$ for every $\rho' \Ge \rho.$ In view of (i), this is not possible, and hence we get (ii).
\end{proof}

We also need a solution to the Hilbert's seventh problem in the proof of Theorem~\ref{MA-TH} (see \cite[Chapter~10]{Ni} and \cite{Tu} for more details).
\begin{lemma} \label{Gelfond-S}
	Let $\omega$ be an algebraic and irrational real number, and let $\psi: \mathbb N \times \mathbb N \rar \mathbb R$ be a function defined by $$\psi(m,n) = \log(mn^{\omega}), \quad m, n \in \mathbb N.$$ Then $\psi$ is an injective function and the range of $\psi$ is a countably infinite subset of $\mathbb R.$
\end{lemma}
\begin{proof}
	If $\psi(m, n) = \psi(p, q)$ for some $(m,n), (p,q) \in \mathbb N \times \mathbb N,$ then since logarithmic function on $(0, \infty)$ is injective, we obtain 
	\beq \label{CON}
	\frac{m}{p} = \Big(\frac{q}{n}\Big)^{\omega}.
	\eeq
Recall that the Gelfond-Schneider Theorem (see \cite[Theorem~10.1]{Ni}) asserts that  if $a$ and $b$ are algebraic numbers with $a \neq 0,$ $a \neq 1,$ and if $b$ is not a real rational number, then any value of $a^b$ is transcendental.
If $q \neq n,$ then applying this theorem to $a = \frac{q}{n} \notin \{0, 1\}$ and $b=\omega,$ we obtain	 that $\big(\frac{q}{n}\big)^{\omega}$ is a transcendental number, which contradicts \eqref{CON}. Thus $q = n,$ and hence by another application of 
\eqref{CON},
$m=p,$ which proves that $\psi$ is injective. The remaining part is now clear since $\mathbb N \times \mathbb N$ is countably infinite. 
\end{proof}

 We also need a characterization of $G$-quasi-invariant kernels $\kappa$ with respect to some $\mathcal J : G \times \Omega \rar \mathbb{C}.$ To state this result, consider the map $U_{g, \mathcal J}$ given by 
\beq \label{U-psi-J} 
U_{g, \mathcal J}(f)(s) = \mathcal J(g^{-1}, s) f(g^{-1} \cdot s), \quad f \in \text{Hol}(\Omega),
\eeq 
where $g \in G$ is fixed but arbitrary.  
 For the case of $\Omega = \mathbb D$ and $G = \mathrm{Aut}(\mathbb D),$ this characterization appears in \cite[Proposition~2.1]{KM}, and for the general bounded domain $\Omega$ with $G = \mathrm{Aut}(\Omega),$ this is stated in \cite[Proposition 6.6]{GmGh}. The proof given in \cite{KM} extends naturally to even unbounded domains. 
 
\begin{theorem} \label{CH-QU} 
Let $\Omega$ be a domain and let $G$ be a group acting  on $\Omega.$
Let $\kappa: \Omega \times \Omega \rar \mathbb{C}$ be a positive semi-definite kernel and let $\mathscr H(\kappa)$ be the reproducing kernel Hilbert space associated with $\kappa.$
	Then $\kappa$ is $G$-quasi-invariant with respect to $\mathcal J$ if and only if the linear map $U_{g, \mathcal J},$ as given in \eqref{U-psi-J}, is unitary on $\mathscr H(\kappa)$ for all $g \in G.$
\end{theorem}

\begin{proof}[Proof of Theorem~\ref{MA-TH}] 
Clearly, the zero kernel function defined on any half plane $\mathbb H_\rho$ satisfies \eqref{QU-DE} for $\mathcal J(\psi, .) = 1, \psi \in \mathrm{Aut}(\mathbb H_{\rho}),$ which yields that $\kappa_\mf a$ is quasi-invariant with respect to $\mathcal J.$ 
Also, in this case, (iii) holds with $f=0.$ 
Hence, 
we may assume that $\kappa_{\mf a}$ is non-zero. 
Also, since (i)$\Rightarrow$(ii) is trivial, it suffices to check that (ii)$\Rightarrow$(iii),  (iii)$\Leftrightarrow$(iv) and  (iii)$\Rightarrow$(i).

(ii)$\Rightarrow$(iii): 
	Suppose that $\kappa_\mf a$ is quasi-invariant with respect to some $\mathcal J.$ 
	By Theorem~\ref{CH-QU} ,  $U_{\psi, \mathcal J}$ is unitary on $\mathscr H_\mf a$ for every $\psi \in \mathrm{Aut}_L(\mathbb H_\rho).$ By Theorem~\ref{char-pdsk},
the analytic symbol $A_{n, \mf a} \in \mathscr H_{\mf a}$ for every integer $n \Ge 1.$ Let $i, j$ be two positive integers.	
	Since $U_{\psi, \mathcal J}(A_{i, \mf a})$ and $U_{\psi, \mathcal J}(A_{j, \mf a})$ belong to $\mathscr H_{\mf a},$ by Theorem~\ref{total}(ii), we obtain sequences $\{c_n\}_{n \Ge 1}$ and $\{d_n\}_{n \Ge 1}$ such that  
	\beq \label{cn-dn}
	U_{\psi, \mathcal J}(A_{i, \mf a})(s) = \sum_{n =1}^{\infty} c_n n^{-s}, \quad U_{\psi, \mathcal J}(A_{j, \mf a})(s) = \sum_{n =1}^{\infty} d_n n^{-s}, s \in \mathbb H_\rho.
	\eeq
By \eqref{U-psi-J},
\beq \label{inter-U-phi-A}
U_{\psi, \mathcal J}(A_{i, \mf a}) A_{j, \mf a} \circ \psi^{-1} =  U_{\psi, \mathcal J}(A_{j, \mf a})A_{i, \mf a}\circ \psi^{-1}, \quad \psi \in \mathrm{Aut}(\mathbb H_\rho).
\eeq
For a positive algebraic and irrational number $\omega$ (for example, $\omega=\sqrt{2}$ will serve the purpose), let $\psi(s) = \frac{1}{\omega}(s-\rho)+\rho,$ $s \in \mathbb H_\rho.$ Since
$\psi^{-1}(s)= \omega (s-\rho)+\rho,$ $s \in \mathbb H_\rho,$
by \eqref{cn-dn} and \eqref{inter-U-phi-A}, we have
	\beq \label{EX-IR}
	 \left.
 \begin{array}{ccc}
\displaystyle \Big(\sum_{n =1}^{\infty} c_n n^{-s}\Big) \Big(\sum_{n = 1}^{\infty} a_{n,j}  n^{(\omega -1)\rho} n^{-\omega s}\Big) \\
= \displaystyle \Big(\sum_{n =1}^{\infty} d_n n^{-s}\Big) \Big(\sum_{n = 1}^{\infty} a_{n,i} n^{(\omega-1)\rho} n^{-\omega s}\Big), \quad s \in \mathbb H_\rho.
 \end{array}
\right\}
\eeq
We may conclude from Remark \ref{D-S-rmk-0} that all the general Dirichlet series appearing in \eqref{EX-IR} are absolutely convergent on $\mathbb H_{\rho'},$ where $\rho'=\rho + \max\{1, \omega^{-1}\}.$ 
Since
$\varphi(m,n) = \log(m)+\omega \log(n)$ is injective on $\mathbb N \times \mathbb N$ (see Lemma \ref{Gelfond-S}), 
an application of Lemma \ref{PR-GE-DI} 
with $\lambda_m = \log(m)$ and $\mu_n = \omega \log(n)$  yields that 
\eqref{EX-IR} may be rewritten as
	 \beqn
		\sum_{q = 1}^{\infty} c_{m_q} a_{n_q,j}n_q^{(\omega-1)\rho}e^{-\nu_qs} = \sum_{q = 1}^{\infty}d_{m_q} a_{n_q,i}n_q^{(\omega-1)\rho}e^{-\nu_qs}, \quad s \in \mathbb H_{\rho'},
	 \eeqn 
	 where $\nu_q = \lambda_{m_q} + \mu_{n_q}.$
By Proposition~\ref{uds},
\beq  
\label{recurrence}
c_{m_q} a_{n_q,j} = d_{m_q} a_{n_q,i} ~\mbox{for all integers}~ q \Ge 1.
\eeq
Let $\{\nu_{q_k}\}_{k \Ge 1}$ be the sequence $\{\varphi(k, i)\}_{k \Ge 1}.$
Since $\varphi$ is injective, $(k, i) = \varphi^{-1}(\nu_{q_k})=(m_{q_k}, n_{q_k})$ for every integer $k \Ge 1.$ Letting this in \eqref{recurrence}, we get
$c_k a_{i,j} = d_k a_{i,i}$ for every integer $k \Ge 1.$ 
If $a_{i, i}=0,$ then by \eqref{eq-gram}, $A_{i, \mf a}=0.$ Otherwise, we obtain 
	  \beqn
	  d_k = \frac{a_{i,j}}{a_{i,i}}c_k, \quad \mbox{for every integer}~k \Ge 1.
	  \eeqn
	   By \eqref{cn-dn}, it follows that $U_{\psi, \mathcal J}(A_{i, \mf a})$ and $U_{\psi, \mathcal J}(A_{j, \mf a})$ are linearly dependent. Since $U_{\psi, \mathcal J}$ is unitary, $A_{i, \mf a}$ and $A_{j, \mf a}$ are linearly dependent. Since $i, j \in \mathbb N$ are fixed but arbitrary and $\{A_{q, \mf a}\}_{q \Ge 1}$ is a total set in $\mathscr H_\mf a$ (see Theorem~\ref{total}(iii)), the dimension of $\mathscr H_\mf a$ is exactly $1.$ Since $\dim(\mathscr H_\mf a) = 1,$ $\mathscr H_\mf a$ is spanned by some function $g$ of unit norm. Thus, $\kappa_\mf a(\cdot, s) = c(s)g$ for some scalar $c(s),$ and hence define $f:\mathbb H_\rho \rar \mathbb C$ by $f(s) = \overline{c(s)}.$ Thus, by \eqref{rp},
			\beqn 
					\kappa_\mf a(s,u) = \langle \kappa_\mf a(.,u), \kappa_\mf a(.,s) \rangle &= c(u)\overline{c(s)} = f(s)\overline{f(u)}, \quad s, u \in \mathbb H_\rho.
\eeqn	   
Since $g$ is non-zero, $g(u_0) \neq 0$ for some $u_0 \in \mathbb H_\rho.$ Thus, $f(\cdot)=\frac{\kappa_{\mf a}(\cdot, u_0)}{\overline{g(u_0)}}$ is a Dirichlet series. Since $\kappa_{\mf a}$ is nowhere-vanishing (see Lemma~\ref{EX-PR}(ii)), so is $f.$ 

(iii)$\Leftrightarrow$(iv): In view of the discussion in the last paragraph, this follows from \cite[Proposition~2.19]{PR}. 
	   
	(iii)$\Rightarrow$(i):  
Since $\kappa_{\mf a}$ is nonzero by assumption,  $f$ is nowhere-vanishing. Define $\mathcal J: \mathrm{Aut}(\mathbb H_{\rho}) \times \mathbb H_{\rho} \rar \mathbb{C}$ by 
\beqn
\mathcal J(\psi, s) = \frac{f(s)}{f(\psi(s))}, \quad \psi \in \mathrm{Aut}(\mathbb H_{\rho}), \, s \in \mathbb H_\rho.
\eeqn
Since $f$ is a holomorphic function and $\psi \in \mathrm{Aut}(\mathbb H_\rho),$ $\mathcal J(\psi, \cdot)$ is holomorphic for each $\psi \in \mathrm{Aut}(\mathbb H_\rho).$ Also, 
		\beqn
		\mathcal J(\psi,s ) \kappa_{\mf a}(\psi(s), \psi(u)) \overline{\mathcal J(\psi, u)} = \kappa_{\mf a}(s,u), \quad s, u \in \mathbb H_{\rho}, \psi \in \mathrm{Aut}(\mathbb H_\rho).
		\eeqn
		Thus, $\kappa_{\mf a}$ is quasi-invariant with respect to $\mathcal J.$ 	
\end{proof}

The following is a manifestation of the fact that Bergman-type kernels are never invariant under the automorphism group (see \cite[Section~3]{Mi}).
\begin{corollary} 
\label{no-aut-linear}
Let $\mathbf a = (a_{m, n})_{m,n = 1}^{\infty}$ be a matrix with complex entries and $\kappa_{\mf a}  : \mathbb H_{\rho} \times \mathbb H_{\rho} \rar \mathbb C$ be the Dirichlet series kernel with the coefficient matrix $\mf a.$ If $\kappa_{\mf a}$ is non-constant, then 
 it is never
$\mathrm{Aut}_L(\mathbb H_\rho)$-invariant.
\end{corollary}
\begin{proof}
	Assume that $\kappa_{\mf a}$ is $\mathrm{Aut}_L(\mathbb H_\rho)$-invariant. Without loss of generality, we may assume that $\kappa_{\mf a}$ non-zero. 
Thus, $\kappa_\mf a$ is $\mathrm{Aut}_L(\mathbb H_\rho)$-quasi-invariant with respect to $\mathcal J(\psi, s) = 1, ~s \in \mathbb H_\rho, ~\psi \in \mathrm{Aut}_L(\mathbb H_\rho).$	
Hence, by Theorem~\ref{MA-TH},
	there exists a nowhere-vanishing Dirichlet series $f:\mathbb H_\rho \rar \mathbb C$ such that 
	\beq \label{form-kernel}			
	\kappa_\mf a(s,u) =  f(s)\overline{f(u)}, \quad s, u \in \mathbb H_\rho.
\eeq
Moreover, by Theorem~\ref{total}(ii), $f(s) = \sum_{n = 1}^{\infty} c_n n^{-s}, s \in \mathbb H_\rho,$ where $c_n=\langle f, A_{n, \mf a} \rangle.$  
Since $\kappa_{\mf a}$ is invariant under translation automorphisms $\phi_b$ (see \eqref{phi-b}), by \eqref{form-kernel}, we obtain
	\beqn
	\sum_{m, n=1}^{\infty} c_m \overline{c_n} m^{ib} n^{-ib} m^{-s} n^{-\bar{u}} =\sum_{m, n=1}^{\infty} c_m \overline{c_n} m^{-s} n^{-\bar{u}} , \quad b \in \mathbb{R}, ~s, u \in \mathbb H_\rho.
	\eeqn		
	By Lemma~\ref{D-S-rmk}(iii), we obtain
	\beqn
	c_m \overline{c_n}(m^{ib} n^{-ib} - 1) = 0, \quad m \neq n, ~ b \in \mathbb{R},
	\eeqn
	or equivalently,	$c_m \overline{c_n} = 0$ for every pair of integers $m,n$ such that $m \neq n$. This implies that the cardinality of $\{m \in \mathbb N: c_m \neq 0\}$ is either $0$ or $1.$ Since $f\neq 0,$  the cardinality of $\{m \in \mathbb N: c_m \neq 0\}$ must be $1.$ If $c_j \neq 0$ for some $j \in \mathbb N,$  then
	\beq \label{FO-CO}
	\kappa_\mf a(s,u) = c_j^2 j^{-s-\overline{u}},~s, u \in \mathbb H_\rho.
	\eeq
	Then by the invariance of $\kappa_{\mf a},$ for $\psi(s) = a^2(s-\rho) + \rho,$ \,$a \in \mathbb R \backslash \{0\},$ we obtain  
	\beqn
	j^{-a^2(s+\overline{u})+2\rho(a^2-1)} = j^{-s-\overline{u}},~s, u \in \mathbb H_\rho.
	\eeqn 
	In particular, for $u = 2\rho,$ this implies that $j^{(1-a^2)s} = 1, s \in \mathbb H_\rho.$ So, $j$ must be $1,$ which by \eqref{FO-CO} yields that $\kappa_\mf a$ is constant. This completes the proof.
\end{proof}

\section{$\mathscr T$-homogeneous operators} \label{5}

 Let $\mathbf a = (a_{m, n})_{m,n = 1}^{\infty}$ be a matrix with complex entries and $\kappa_{\mf a}  : \mathbb H_{\rho} \times \mathbb H_{\rho} \rar \mathbb C$ be the positive semi-definite Dirichlet series kernel with the coefficient matrix $\mf a.$ Let $\phi_b$ be given by 
	\eqref{phi-b} and 
	 let $\mathscr T =\{\phi_b : b \in \mathbb R\}$ be the subgroup of $\mathrm{Aut}(\mathbb H_\rho).$  
	 We say that $\kappa_{\mf a}$ is {\it translation-invariant} if it is $\mathscr T$-invariant, that is, 
		\beqn
		\kappa_{\mf a}(\phi_b(s), \phi_b(u)) = \kappa_{\mf a}(s, u), \quad b \in \mathbb R, ~s, u \in \mathbb H_\rho.
		\eeqn

It is easy to characterize all translation-invariant Dirichlet series kernels.
	\begin{proposition} \label{prop-t-inv}
Let $\mathbf a = (a_{m, n})_{m,n = 1}^{\infty}$ be a matrix with complex entries and $\kappa_{\mf a}  : \mathbb H_{\rho} \times \mathbb H_{\rho} \rar \mathbb C$ be the positive semi-definite Dirichlet series kernel with the coefficient matrix $\mf a.$ 	
	Then  $\kappa_{\mf a}$ is translation-invariant if and only if 
$\mf a$ is a diagonal matrix.
	\end{proposition}
	\begin{proof} We argue as in the proof of Corollary~\ref{no-aut-linear}. 
Note that for any $s, u \in \mathbb H_\rho$ and $b \in \mathbb R,$
\beqn
\kappa_{\mf a}(\phi_{b}(s), \phi_{b}(u)) &=& 
\sum_{m = 1}^{\infty} a_{m,m} m^{-\phi_{b}(s)- \overline{\phi_{b}(u)}} + \sum_{m \neq n}^{\infty} a_{m, n} m^{-\phi_{b}(s)} n^{- \overline{\phi_{b}(u)}} \\
&=& \sum_{m = 1}^{\infty} a_{m,m}m^{-s-\overline{u}} + \sum_{m \neq n}^{\infty} a_{m, n} m^{ib} n^{-ib} m^{-s} n^{-\bar{u}}. 
\eeqn
Thus, $\kappa_{\mf a}(\phi_{b}(s), \phi_{b}(u)) = \kappa_{\mf a}(s, u)$ holds for $b \in \mathbb{R},$ $~s, u \in \mathbb H_\rho$ if and only if 
\beqn
 \sum_{m \neq n}^{\infty} a_{m, n} m^{ib} n^{-ib} m^{-s} n^{-\bar{u}} =\sum_{m \neq n}^{\infty} a_{m, n} m^{-s} n^{-\bar{u}} , \quad b \in \mathbb{R}, ~s, u \in \mathbb H_\rho.
\eeqn		
By Lemma~\ref{D-S-rmk}(iii), this is equivalent to 
		$$a_{m, n}(m^{ib} n^{-ib} - 1) = 0, \quad m \neq n, ~ b \in \mathbb{R},$$
or equivalently,	$a_{m, n} = 0$ for every $m \neq n;$ that is, $\kappa_\mf a$ is a diagonal Dirichlet series kernel. This completes the proof.
	\end{proof}

Consider the diagonal matrix $\mathrm{diag}(\mf a)$ with diagonal entries $\mf a:=\{a_n\}_{n = 1}^{\infty}$ of positive real numbers. Let $\kappa_{\mathrm{diag}(\mf a)}$ be the Dirichlet series kernel with the coefficient matrix $\mathrm{diag}(\mf a),$ that is,  
\beqn
\kappa_{\mathrm{diag}(\mf a)}(s,u) = \sum_{n = 1}^{\infty} a_n n^{-s-\overline{u}},~ \quad s, u \in \mathbb H_\rho.
\eeqn
For sake of convenience, we denote $\kappa_{\mathrm{diag}(\mf a)}$ by $\kappa_{\mf a},$ and 
we refer to $\kappa_{\mathbf a}$ as the {\it diagonal Dirichlet series kernel associated with $\mf a$}. 
Clearly, the kernel $\kappa_{\mathbf a}$ is a positive definite Dirichlet series kernel (see \cite[Lemma~20]{MS}). 

In the remaining part of this section, we analyze a class of operators acting on the reproducing kernel Hilbert space $\mathscr H_{\mf a}$ associated with a diagonal Dirichlet series kernel $\kappa_{\mf a}.$ We begin with a notion  reminiscent of that of the circularity introduced in \cite{AHHK} (variants of this notion appeared in \cite{BB, KM}).
\begin{definition} \label{t-inv-def} 
Let $\mathcal H$ be a complex Hilbert space. A densely defined linear operator $T$ with domain $\mathcal D(T) \subseteq \mathcal H$ is said to be {\it $\mathscr T$-homogeneous} if for every $c \in \mathbb{R}$, there exists a unitary $U_{c} : \mathcal H \rar \mathcal H$ such that $U_c$ maps $\mathcal D(T)$ into $\mathcal D(T)$ and satisfies 
\beqn
U_{c} Tf = (T-icI) U_{c}f, \quad f \in \mathcal D(T).
	\eeqn	
	\end{definition}
	\begin{remark}
If $T$ is $\mathscr T$-homogeneous and if $\lambda$ is an eigenvalue of $T,$ then for every $c \in \mathbb R,$	$\lambda - ic$ is an eigenvalue of $T.$ Thus, if the point-spectrum $\sigma_p(T)$ of any $\mathscr T$-homogeneous operator $T$ is nonempty, then  
\beqn 
\sigma_p(T) = \{\lambda - ic : \lambda \in \sigma_p(T), \, c \in \mathbb R\}.
\eeqn
Similar remark is valid for the spectrum. 
By \cite[Theorem~3.6]{Co}, the spectrum of  a bounded linear operator $T$ on a nonzero Hilbert space $\mathcal H$ is never empty, and hence any $\mathscr T$-homogeneous operator $T$ in $\mathcal H$ is necessarily unbounded. 
	\end{remark}
	
	For the sake of convenience, we refer to a sequence $\mf a$ of real numbers as {\it admissible sequence} if the support $\mathrm{supp}(\mf a)$ of $\mf a$ is an infinite multiplicative subset of $\mathbb N$ such that there exist positive integers $p, q \in \mathrm{supp}(\mf a) \backslash \{1\}$ with $\gcd(p, q) = 1.$  Clearly, $\{n\}_{n \in \mathbb N}$ is an admissible sequence.
	
	The following result produces a family of $\mathscr T$-homogeneous operators on reproducing kernel Hilbert spaces associated with diagonal Dirichlet series kernels.
	 	\begin{theorem} \label{t-inv-thm}
  		Let $\mathbf a = \{a_{n}\}_{n=1}^{\infty}$ be an admissible sequence of non-negative real numbers and let $\kappa_{\mf a}  : \mathbb H_{\rho} \times \mathbb H_{\rho} \rar \mathbb C$ be the diagonal Dirichlet series kernel associated with $\mf a.$ For $u \in \Omega,$ let $\kappa_{\mf a, u}$ denote the function $\kappa_{\mf a}(\cdot, u)$ in $\mathscr H(\kappa_{\mf a}).$	
  		For every real number $a > \rho,$ define a linear operator $T_a$ on $\mathscr H_{\mf a}$ by
  		\beqn
  		T_a\kappa_{\mf a, a+ib} = ib \,\kappa_{\mf a, a+ib}, \quad b \in \mathbb R,
  		\eeqn
  		and  extend it linearly to $\mathcal D(T_a) :=\mbox{span}\{\kappa_{\mf a, a+ib} : b \in \mathbb R\}.$  
  		Then the following statements are valid$:$
  		\begin{enumerate}
  			\item[(i)] $T_a$ is a $\mathscr T$-homogeneous operator in $\mathscr H_\mf a,$
  			\item[(ii)] the domain $\mathcal D(T^*_a)$ of $T^*_a$ is trivial .
  		\end{enumerate}	
  		In this case, $T_a$ is not closable. 
  	\end{theorem}

	The proof of Theorem~\ref{t-inv-thm} relies on a couple of facts. 
	
	\begin{proposition} \label{LI-INS}
	Let 
 $\mf b =\{b_n\}_{n =1}^{\infty}$ be an admissible sequence of real numbers. 
For every $b \in \mathbb R,$ assume that the Dirichlet series
  		\beqn
  		f_b(s) = \sum_{n \in \mathrm{supp}(\mf b)} n^{ib} b_n n^{-s}
  		\eeqn
  		converges on $\mathbb H_{\rho}$ 
  		for some $\rho \in \mathbb{R}.$ Then the family $\{f_b : b \in \mathbb R\}$ is linearly independent in $\mathrm{Hol}(\mathbb H_\rho).$
  	\end{proposition}
	\begin{proof}
			The proof is divided into two steps. 
			
			{\bf Step 1:}  We prove by induction on $N \Ge 1$ that 
		\begin{enumerate} \label{claim-lin-in.}
		\item[]	if there exist $c_{1}, \ldots, c_{N} \in \mathbb{C}$ and distinct numbers $d_{1}, \ldots, d_N \in \mathbb{R}$ such that $\sum_{j = 1}^{N} c_{j} n^{id_{j}} = 0$ for every integer $n \in \mathrm{supp}(\mf b),$
		then 
$c_j =0$ for every $j=1, \ldots, N.$
\end{enumerate}

		If $N =1$, then this is trivial. Assume that \eqref{claim-lin-in.} holds for $N.$ Suppose to the contrary, that \eqref{claim-lin-in.} does not hold for $N+1$. After renaming and the coefficients $c_1, \ldots, c_{N+1}$ and distinct real numbers $d_{1}, \ldots, d_{N+1},$ we may assume that 
		\beq \label{repre-nid0}
		n^{i d_{0}} = \sum_{j = 1}^{N} c_{j} n^{id_{j}}, \quad n \in \mathrm{supp}(\mf b).
		\eeq 
Thus, there exists $l \in \{1, \ldots,N\}$ such that $c_{l} \neq 0.$ For integers $m, n \in \mathrm{supp}(\mf b)$, by \eqref{repre-nid0} and the multiplicativity of $\mathrm{supp}(\mf b),$
		\beqn 
		\sum_{j = 1}^{N} c_{j} m^{id_{j}}n^{id_{j}} = (mn)^{id_{0}} = m^{id_{0}}n^{id_{0}} = \sum_{j = 1}^{N} c_{j} m^{id_{0}}n^{id_{j}},
		\eeqn
		and hence we obtain
		\beqn
		\sum_{j = 1}^{N} c_{j} (m^{id_{0}} - m^{id_{j}}) n^{id_{j}} = 0, \quad m, n \in \mathrm{supp}(\mf b).
		\eeqn
		By the induction hypothesis, we have $c_{j} (m^{id_{0}} - m^{id_{j}}) = 0$ for $j =1,\ldots, N$ and integers $m \in \mathrm{supp}(\mf b).$ Since $c_l \neq 0,$ we get 
		$m^{id_{0}} = m^{id_{l}}$ for every integer $m \in \mathrm{supp}(\mf b).$
		Thus, for every integer $m \in \mathrm{supp}(\mf b),$ there exists a nonzero integer $k_{m}$ such that
\beq \label{letting-m}
(d_{0} - d_{l}) \log(m) = 2 k_{m} \pi.
\eeq
Since $\mf b$ is admissible, $\mathrm{supp}(\mf b)$ contain integers $p, q \neq 1$ with $\gcd(p, q)=1.$ Thus, by \eqref{letting-m}, we obtain
		$\frac{\log(p)}{\log(q)} = \frac{k_{p}}{k_{q}},$
		which is not possible since $\frac{\log(p)}{\log(q)}$ is an irrational number (consult the Fundamental Theorem of Arithmetic). Hence, \eqref{claim-lin-in.} holds true for $k = N+1.$

		{\bf Step 2:} To prove the linear independence of the family $\{f_b : b \in \mathbb R\},$ suppose that there exists $\{c_{j}\}_{j = 1}^{N} \subseteq \mathbb{C}$ and $\{d_{j}\}_{j=1}^{N} \subseteq \mathbb{R}$ such that 
		\beqn
		\sum_{j = 1}^{N} c_{j} f_{d_{j}}(s)=0, \quad s \in \mathbb H_\rho.
		\eeqn
		Thus, we obtain 
		\beqn
		 \sum_{n \in \mathrm{supp}(\mf b)} \Big(\sum_{j = 1}^{N} c_{j}  n^{id_j}\Big) b_n n^{-s}=0, \quad s \in \mathbb H_\rho.
		\eeqn
Since each $b_n \neq 0$ for $n \in \mathrm{supp}(\mf b),$ by Proposition~\ref{uds},  we obtain 
		\beqn
		\sum_{j = 1}^{N} c_{j} n^{id_{j}} = 0, \quad n \in \mathrm{supp}(\mf b).
		\eeqn
		Thus, by Step~1, we obtain that 
		\beqn		
		c_j =0, \quad j=1, \ldots, N.
		\eeqn
		This completes the proof.
	\end{proof}
\begin{corollary} 
\label{translate-k}		
Under the assumptions of Theorem~\ref{t-inv-thm},  
	for every real number $a > \rho,$ the family $\{\kappa_{\mf a, a+ib} : b \in \mathbb{R}\}$ is a linearly independent total subset of $\mathscr H_{\mf a}.$
\end{corollary}
\begin{proof}
Define $\mf b$ by
\beqn 
b_n = \begin{cases} a_{n} n^{-a} & \mbox{if}~n \in  \mathrm{supp}(\mf a), \\
0 & \mbox{otherwise}.
\end{cases}
\eeqn
Applying Proposition ~\ref{LI-INS} to $\mf b,$ 
we obtain the linear independence of $\{\kappa_{\mf a, a+ib} : b \in \mathbb{R}\}.$ The remaining part follows from the reproducing property \eqref{rp} and an application of the identity theorem.
\end{proof}
As a consequence of Corollary~\ref{translate-k}, we conclude that the vertical translates of the Riemann zeta function are linearly independent in $\mathscr H_{\mf 1}.$ We also need the following lemma in the proof of Theorem~\ref{t-inv-thm}. 
	\begin{lemma} \label{suff-t-inv}
	Let $\mathbf a = \{a_{n}\}_{n = 1}^{\infty}$ be a sequence of nonnegative real numbers and $\kappa_{\mf a}  : \mathbb H_{\rho} \times \mathbb H_{\rho} \rar \mathbb C$ be the diagonal Dirichlet series kernel associated with $\mf a.$  For $b \in \mathbb R$ and an admissible sequence $\mf b =\{b_n\}_{n=1}^{\infty}$ of real numbers such that $\mathrm{supp}(\mf b) = \mathrm{supp}(\mf a)$, consider the Dirichlet series
	\beqn
	f_b(s) = \sum_{n \in \mathrm{supp}(\mf b)} n^{ib} b_n n^{-s}, \quad s \in \mathbb H_\rho.
	\eeqn
	Assume that the family $\{f_b : b \in \mathbb R\}$ is a total subset of $\mathscr H_{\mf a}.$ Define a linear operator $T$ on $\mathscr H_{\mf a}$ by
	\beqn
	Tf_b = ib f_b, \quad b \in \mathbb R,
	\eeqn
	and  extend it linearly to $\mathcal D(T)=\mbox{span}\{f_b : b \in \mathbb R\}.$ Then 
	$T$ is a $\mathscr T$-homogeneous operator in $\mathscr H_\mf a.$ If, in addition, there exists a positive number $\delta > 0$ such that 
	\beq \label{Mdelta}
	\sum_{n \in \mathrm{supp}(\mf a)} \frac{n^{2\delta}b^2_n}{a_n}  < \infty,
	\eeq
	 then the domain $\mathcal D(T^*)$ of the Hilbert space adjoint $T^{*}$ of $T$ is trivial.
	\end{lemma}
	\begin{proof} By Proposition ~\ref{LI-INS}, $T$ is well defined and by the assumption, $T$ is a densely defined operator.
	Let $c \in \mathbb R$ and define a unitary $U_c$ by setting $U_c(n^{-s})=n^{ic} n^{-s},$ $n \in \mathrm{supp}(\mf a),$ and extending it linearly and continuously to $\mathscr H_{\mf a}.$  
Note that  
\beq 	\label{action-Uc}
	(U_c f_b)(s) = \sum_{n \in \mathrm{supp}(\mf a)} n^{ib} b_n U_c(n^{-s}) = f_{b+c}(s),  \quad b, c \in \mathbb R, \, s \in \mathbb H_\rho,
\eeq	
and hence	$\mathcal D(T)$ is invariant under $U_c.$ It follows that for any $b, c \in \mathbb R,$ 
	\beq \label{trans-inv}
	U_c T f_b = ib U_c f_b \overset{\eqref{action-Uc}}= ib f_{b+c},
	\eeq
	and consequently,
	\beqn
	(T-icI) U_{c}f_b &\overset{\eqref{action-Uc}}=& (T-icI)f_{b+c} \\
	&=& i(b+c)f_{b+c} - ic f_{b+c} \\
	& \overset{\eqref{trans-inv}}=& U_c T f_b.
	\eeqn
Thus, $T$ is a $\mathscr T$-homogeneous operator in $\mathscr H_\mf a.$
	
	Assume that \eqref{Mdelta} holds for some $\delta > 0,$ and let $h \in \mathcal D(T^*).$ Thus, there exists a Dirichlet series $g_h \in \mathscr H_{\mf a}$ such that 
	\beqn
-ib\inp{h}{f_b} = \inp{h}{Tf_b}=\inp{g_h}{f_b}, \quad b \in \mathbb R.
	\eeqn
	It follows that 
\beq \label{phi-1-phi-2}	
	-ib\phi_1(ib) = \phi_2(ib), \quad b \in \mathbb R,
\eeq	
where $\phi_1$ and $\phi_2$ are formal Dirichlet series given by
		\beqn
		\phi_{1}(s) = \sum_{n \in \mathrm{supp}(\mf a)} \frac{1}{a_n} \hat{h}(n)b_n n^{-s}, ~ \phi_{2}(s) = \sum_{n \in \mathrm{supp}(\mf a)} \frac{1}{a_n}\hat{g_h}(n)b_n n^{-s}
		\eeqn
		with $\hat{h}(\cdot)$ and $\hat{g_h}(\cdot)$ denoting the Dirichlet coefficients of the Dirichlet series $h$ and $g_h,$ respectively.
		By the Cauchy-Schwarz inequality,
		\beqn	
		\Big|\sum_{n \in \mathrm{supp}(\mf a)} \frac{1}{a_n} \hat{h}(n)b_n n^{\delta}\Big|   & \Le &  \sum_{n \in \mathrm{supp}(\mf a)} \frac{1}{a_n} |\hat{h}(n)| b_n n^{\delta} \\
	&  \Le & \|h\|_{\mathscr H_\mf a} \Big(\sum_{n \in \mathrm{supp}(\mf a)} \frac{n^{2\delta}b^2_n}{a_n}\Big)^{1/2},
	\eeqn
		which is finite by \eqref{Mdelta}.
		Thus, $\phi_{1}$ is absolutely convergent on $\mathbb H_{-\delta}.$ A similar argument shows that $\phi_2$ is also absolutely convergent on $\mathbb H_{-\delta}.$ Hence $s\phi_{1}(s)$ and $\phi_{2}(s)$ are holomorphic functions on $\mathbb H_{-\delta}.$ Since $-s\phi_{1}(s)$ and $\phi_{2}(s)$ agree on the imaginary axis (see \eqref{phi-1-phi-2}), by the identity theorem (see \cite{ShSt}), 
		\beq \label{phi-1-phi-2-second}	
		-s\phi_{1}(s) = \phi_{2}(s), \quad s \in \mathbb H_{-\delta}. 
		\eeq
We claim that $\hat{\phi}_1(n) = 0$ and $\hat{\phi}_2(n) = 0$ for every integer $n \Ge 1$ (although we need to check this for every $n \in \mathrm{supp}(\mf a),$ we find it convenient to check the same for all positive integers $n$), where $\hat{\phi}_j(n)$ denotes the $n$th Dirichlet coefficient of $\phi_j, j = 1,2.$ First we show that $\hat{\phi}_1(1) = 0.$ Indeed, for $s \in \mathbb H_{0},$ by \eqref{phi-1-phi-2-second},
		\beqn 
		 \lim_{s \rar \infty} \phi_{1}(s) &=& -\lim_{s \rar \infty} \frac{\phi_{2}(s)}{s} = 0,\eeqn
		 which yields $\hat{\phi}_1(1) = 0.$ 
		 Hence, we obtain,
		 \beqn
		   \lim_{s \rar \infty} \phi_{2}(s)  = -\lim_{s \rar \infty} s \phi_{1}(s) &=& -\lim_{s \rar \infty} \frac{s}{2^{s}} \lim_{s \rar \infty}2^{s}\phi_{1}(s) = 0,
		  \eeqn
		which yields that $\hat{\phi}_2(1) = 0.$ 
		Assuming that $\hat{\phi}_1(k) = 0$ and $\hat{\phi}_2(k) = 0$ for $k=1, \ldots, n-1,$ we get
		 \beqn
		  \sum_{m=n}^{\infty} \hat{\phi}_1(m) \Big(\frac{m}{n}\Big)^{-s} = 
		  -\frac{1}{s}\sum_{m=n}^{\infty}\hat{\phi}_2(m)\Big(\frac{m}{n}\Big)^{-s}, \quad s \in \mathbb H_{0},
		  \eeqn
		 and hence letting $s \rar \infty,$ we obtain $\hat{\phi}_1(n) = 0$ and $\hat{\phi}_2(n) = 0.$
		 Since $\hat{\phi}_1(n) = \frac{1}{a_n} \hat{h}(n)b_n$ for every integer $n \Ge 1,$ it follows that $h=0.$  
This shows that $\mathcal D(T^*) = \{0\}.$
 		\end{proof}
  	
  	\begin{proof}[Proof of Theorem~\ref{t-inv-thm}]
  	Part (i)  follows from Corollary~\ref{translate-k} and Lemma~\ref{suff-t-inv}. To see (ii), let $a > \rho$ and choose $\delta > 0$ such that $a-\delta> \rho.$ For $b_n =  a_{n}n^{-a},$ $n \in \mathrm{supp}(\mf a),$ note that
  	\beqn
  	\sum_{n \in \mathrm{supp}(\mf a)} \frac{n^{2\delta}b_n^{2}}{a_n} = \sum_{n \in \mathrm{supp}(\mf a)} a_n n^{-2(a-\delta)} = \kappa_\mf a(a-\delta, a-\delta) < \infty.
  	\eeqn
It now follows from Lemma~\ref{suff-t-inv} that $\mathcal D(T^*)$ is trivial. Since a densely defined linear operator is closable if and only if its adjoint is densely defined (see \cite[Theorem~1.8(i)]{Sh}), the remaining part follows from (ii).
  	\end{proof}
  	
  	We do not know whether or not there exists a closed $\mathscr T$-homogeneous operator on a reproducing kernel Hilbert space associated with a Dirichlet series kernel.
  	
 \vskip.2cm

 \textit{Acknowledgment:}
The authors convey their sincere thanks to Md. Ramiz Reza, Soumitra Ghara and Rahul Singh for some helpful discussions concerning the subject of the paper.
\vskip.3cm

\textit{Declarations:}

\vskip.2cm

\noindent
{\bf Conflict of interest} The authors declare that they have no conflict of interest.

{}
\end{document}